\documentclass[11pt]{article}

\usepackage{amsmath,amsthm,verbatim,amssymb,amsfonts,amscd, graphicx}
\usepackage{graphics}
\theoremstyle{plain}

\newtheorem{question}{Question} 
\newtheorem{thm}{Theorem}
\newtheorem{defn}{Definition}[section]

\newtheorem{prop}[defn]{Proposition}
\newtheorem{remark}[defn]{Remark}
\newtheorem{lm}[defn]{Lemma}
\newtheorem{case}{Case}
\newtheorem{ep}[defn]{Example}
\usepackage{extarrows}
\usepackage{mathrsfs}
\usepackage{galois}
\usepackage[colorlinks,citecolor = red, linkcolor=blue,hyperindex,CJKbookmarks]{hyperref}
\usepackage[all]{hypcap}
\usepackage[all]{xy}
\usepackage{tikz-cd}
\usepackage{latexsym}

\begin{document}
	
	\title{The extending surfaces of immersions into surfaces}
	\author{Bojun Zhao}
	\maketitle
	
	ABSTRACT: S. Blank solved the question of classifying immersed circles in $\mathbb{R}^{2}$ that extend to immersed disks, and how many topologically inequivalent disks can be extended. The quetions of various cases in $2$-dimension have already been solved by generalizing his method. In this paper, we give a new way, which is straightforward for the questions, and we determine all topological equivalence classes of immersed surfaces bounded by an arbitrary immersed circle in a closed oriented surface.
	
	\tableofcontents
	
	\section{Introduction}\
	
An immersion of a circle into a surface is \emph{normal} if its image is in general position,
i.e. it has finitely many self-intersections and each self-intersection is a transverse double point (\emph{node}).
In this paper,
we assume all immersions are smooth are oriented.
For an immersion $f: S^{1} \to \Sigma$ ($\Sigma$ is a surface),
we will always assume $f$ is normal.
For an immersion $g: S \to \Sigma$ ($S,\Sigma$ are surfaces),
we will always assume $g \mid_{\partial S}$ is normal.

Fix a surface $\Sigma$ and an immersion $f: S^{1} \to \Sigma$.
We say that $f$ \emph{extends} to an immersion $F: \Sigma_0 \to \Sigma$ ($\Sigma_0$ is a surface) if 
$F \mid_{\partial \Sigma_0} = f$.
We will always assume the interior of $\Sigma_0$ is mapped to the left of $f(S^{1})$ by $F$.

\begin{defn}\rm
	Let $S$ and $\Sigma$ be surfaces with $\partial S \ne \emptyset$.
	$g_1,g_2: S \to \Sigma$ are $2$ immersions such that
	$g_1 \mid_{\partial S} = g_2 \mid _{\partial S}$.
	$g_1,g_2$ are (topologically) \emph{equivalent} if there exists a diffeomorphism $h: S \to S$
	such that $h \mid_{\partial S} = id$ and $g_1 = g_2 \circ h$ (see \hyperref[Curley, Wolitzer]{[4]}).
\end{defn}

The questions are to classify the immersions that can be extended, 
and to give the number of topologically inequivalent extensions they have (or, determine all topological equivalence classes of extensions).
To answer the questions,
we should solve both the existence problems and the equivalence problems (see \hyperref[Spring]{[14]}).
We list the questions:

\begin{question}\label{question 1}\rm 
	Which immersed circles in $\mathbb{R}^{2}$ can be extended to immersed disks,
	and in how many inequivalent ways?
	(see \hyperref[Spring]{[14]})
\end{question}

\begin{question}\label{question 2}\rm 
	Which immersed circles in $\mathbb{R}^{2}$ can be extended to immersed surfaces,
	and in how many inequivalent ways?
	(see \hyperref[Bailey]{[1]})
\end{question}

\begin{question}\label{question 3}\rm 
	Given a closed oriented surface $\Sigma$ and a (homologically trivial) immersion $f:S^{1} \to \Sigma$.
	
	(i) (H.Hopf and R.Thom) When does it extend to an immersed surface? (see \hyperref[McIntyre]{[11]})
	
	(ii) Fix a nonnegative integer $n$,
	how many inequivalent extension $F$ are there such that\\
	$\min_{p \in \Sigma} \#(F^{-1}(p)) = n$?
\end{question}

The questions in $3$-dimension are (samely, the problems ask the existence and equivalence both):

\begin{question}\label{question 4}\rm \hyperref[Kirby]{[9, Problem 3.19]}
	Which immersed $2$-spheres in $\mathbb{R}^{3}$ bound immersed $3$-balls?
	(see \hyperref[Kirby]{[9]})
\end{question}

\begin{question}\label{question 5}\rm 
	Which immersed closed oriented surfaces in $\mathbb{R}^{3}$ bound immersed $3$-manifolds (and in how many inequivalent ways)?
\end{question}

Question \ref{question 1} was solved by S.Blank in his PhD thesis in 1967 (\hyperref[Blank]{[2]};
mentioned in \hyperref[Francis]{[6]},
\hyperref[Poenaru]{[12]}).
Before that,
C.Titus gave the answer of existence problem (\hyperref[Titus]{[13]}).
Blank proved a bijection between the set of topological equivalence classes of immersed disks bounded by an immersed circle in $\mathbb{R}^{2}$ and the set of groupings of a word given by the immmersed circle.
Blank's techniques were generalized widely to deal with other cases.
K. Bailey solved Question \ref{question 2} (\hyperref[Bailey]{[1]}).
G.Francis solved some cases of Question \ref{question 3} (ii) (\hyperref[Francis2]{[7]}).
D. Frisch has claimed to solve Question \ref{question 3} $(ii)$ (\hyperref[Frisch]{[8]}).
Question \ref{question 3} (i) was solved by McIntyre (in the cases of $2$-sphere and closed oriented surfaces of genus greater than $1$) (\hyperref[McIntyre]{[11]}).
There are some other works related to these questions, for example, D.Calegari discussed the rationally bounding (\hyperref[Calegari]{[3]}).

In this paper,
we develop a new way for Question \ref{question 3} (ii),
which is also generalized to address Question \ref{question 5} in the forthcoming paper.
We determine the topological equivalence classes of extensions (to surfaces) of an arbitrary immersed circle in a closed oriented surface.
Different from previous works dealing with different cases in many years,
our treatment is straightforward. 
For a closed oriented surface $\Sigma$ and a (homologically trivial) immersion $f: S^{1} \to \Sigma$,
fix a nonnegative integer $n$,
our main result is to prove a bijection between 
$\{F \mid F$ $is$ $a$ $class$ $of$ $topologically$ $equivalent$ $extensions$ $(to$ $surface)$ $of$ $f,$ $\min_{p \in \Sigma} \#(F^{-1}(p)) = n\}$ and a determined set (obtained by finite operations).
Apply to Question \ref{question 2} (or Question \ref{question 1}),
our conclusion provides the immersed planar circle with a bijection between topological equivalence classes of extensions to surfaces (or disks) and a determined set (obtained by finite operations).

In the forthcoming paper,
we address Question \ref{question 5}.
For an immersion of a closed oriented surface into $\mathbb{R}^{3}$,
we prove a bijection between topological equivalence classes of extensions (to $3$-manifolds) and a set obtained by finite operations.

\subsection{Notations}\

In this paper,
a \emph{graph} means a vertices-edges pair and together with a realization of topological space (an embedded graph). 
If $G$ is a graph,
$v \in V(G)$,
$deg_G(v)$ denotes the degree of $v$ in $G$,
and we say $v$ is a \emph{leaf} of $G$ if $deg_G(v) = 1$.
\emph{Cutting off a set from the space} means to delete the set from the space and do a path compactification.
\emph{Extensions} of an immersed circle means maps of connected compact oriented surfaces with boundary equal to the immersed circle if not otherwise mentioned.
The \emph{equivalence classes} of the extensions imply the topological equivalence classes of the extensions to surfaces if not otherwise mentioned.
A \emph{surface} implies a connected surface if not otherwise mentioned.

\subsection{Main results}\

We explain some basic ingredients first.
Let $\Sigma$ be a surface and $f: S^{1} \to \Sigma$ a homologically trivial immersion.
Assume $\{A_1,\ldots,A_n\}$ is the set of the components of $\Sigma \setminus f(S^{1})$.
$\psi: \{A_1,\ldots,A_n\} \to \mathbb{Z}_{\geqslant 0}$ is a \emph{normal numbering} of $f$ if: at each segment of $f(S^{1})$,
the image of the component in its left is $1$ greater than the image of the component in its right.
Actually,
if $F: \Sigma_0 \to \Sigma$ is an extension of $f$,
then there is a normal numbering $\psi_F$ sending $A_i$ to the cardinality of $F^{-1}(x)$ ($\forall x \in A_i$), $\forall i \in \{1,2,\ldots,n\}$.
We say $\psi_F$ is the normal numbering \emph{given by} $F$,
and $F$ is \emph{related to} $\psi_F$.
Fix a normal numbering $\psi$,
let $D_i(f,\psi)$ ($0 \leqslant i \leqslant \max \psi$) 
be $\bigcup_{k \in \{1,2,\ldots,n\}, \psi(A_k) \geqslant i} \overline{A_k}$ (where $\overline{A_k}$ is the closure of $A_k$).
More generally,
we can extend above definitions to the case of a homologically trivial immersion $f: S^{1} \coprod \ldots \coprod S^{1} \to \Sigma$.
The definitions basically follow from \hyperref[Frisch]{[8]},
\hyperref[McIntyre, Cairns]{[10]}, \hyperref[McIntyre]{[11]}.

Refer to \hyperref[Ezell, Marx]{[5]},
$g: M \to N$ ($M,N$ are compact oriented surfaces, $M$ may be disconnected) is a \emph{polymersion} (Definition \ref{polymersion})
if $g$ is topologically equivalent to $z \mapsto z^{k}$ ($k \in \mathbb{Z}_{\geqslant 1}$) at each $z$ in the interior of $M$, 
and there exists an open set $U$ (if $\partial M \ne \emptyset$)
such that $\partial M \subseteq U$ and
$g \mid_{U}$ is an immersion.

Given a polymersion $g: \Sigma_0 \to \Sigma$,
$\Sigma_0$ is a compact orientable surface ($\Sigma_0$ may disconnected) and $\Sigma$ is a closed oriented surface.
The \emph{cancellation operation} (Definition \ref{cancellation operation})
$(g,\Sigma_0) \leadsto (g_1,\Sigma_1)$
(where $\Sigma_1$ is a compact orientable surface that may be disconnected and
$g_1: \Sigma_1 \to \Sigma$ a polymersion) 
is a transformation defined by some \emph{cancellable domains} (Definition \ref{cancellable components})
$A_1,\ldots,A_n \subseteq \Sigma_0$.
More precisely,
we delete $\stackrel{\circ}{A_1},\ldots,\stackrel{\circ}{A_n}$ from $\Sigma_0$ and identify the segments which have same images.
If the cancellation is \emph{regular} (Definition \ref{regular} (ii)),
then it induces an embedded graph in $\Sigma_1$.

Since the cancellation operation $(g,\Sigma_0) \leadsto (g_1,\Sigma_1)$ depends on the choice of cancellable domains,
we give a way to construct cancellable domains in $2$ cases (see more contents in Subsection \ref{subsection 4.3}).
Case \ref{case 1} yields an embedded graph $G$ (Definition \ref{case 1 graph}),
and $G$ constructs the cancellable domains (Lemma \ref{case 1 cancellation}).
Then $G$ defines a cancellation $(g,\Sigma_0) \stackrel{G}{\leadsto} (g_1,\Sigma_1)$.
Case \ref{case 2} gives an embedded graph $G^{'} \subseteq \Sigma_0$
and yields an embedded graph $G$ (Definition \ref{case 2 graph}).
$G$ and $g(G^{'})$ constructs the cancellable domains (Lemma \ref{case 2 cancellation}).
Then $(G,g(G^{'}))$ defines a cancellation $(g,\Sigma_0) \stackrel{(G,g(G^{'}))}{\leadsto} (g_1,\Sigma_1)$.
More details are given in Section \ref{section 4}.

Fix an immersion $f: S^{1} \to \Sigma$ (where $\Sigma$ is a closed oriented surface) and $\psi$ a normal numbering of $f$.
Given some embedded graphs under certain conditions,
we define a map to extend $f$,
which is a \emph{polymersion} of a surface that may be \emph{disconnected}.
This is called an \emph{inscribed map} of $(f,\psi)$ (Definition \ref{inscribed map}).
We consider a collection of sets of (embedded) graphs to establish the inscribed maps to be \emph{immersions} of \emph{connected surfaces} (then they are extensions of $f$),
and such that different sets establish inequivalent extensions of $f$.

An \emph{inscribed set} $\zeta$ is a finite set obtained by $f$ and $\psi$.
$\forall \{(\tilde{H}_1,H_1),\ldots,(\tilde{H}_n,H_n)\} \in \zeta$,
$\{(\tilde{H}_1,H_1),$ $\ldots,(\tilde{H}_n,H_n)\}$ is called \emph{good} if $H_1 = \emptyset$,
and we denote by $I(\zeta)$ the set of good elements of $\zeta$.
$\forall \{(\tilde{H}_1,H_1),\ldots,(\tilde{H}_n,H_n)\} \in I(\zeta)$,
$H_2,\ldots,H_n$ establish an inscribed map (said to be an inscribed map related to $\{(\tilde{H}_1,H_1),\ldots,(\tilde{H}_n,H_n)\}$),
and we can verify that such inscribed maps are extensions of $f$.
We develop a map $i: I(\zeta) \to E(f,\psi)$ (where $E(f,\psi)$ is the set of equivalence classes of extensions of $f$ related to $\psi$) sending each $X \in I(\zeta)$ to the equivalence class of the inscribed map of $(f,\psi)$ related to $X$.

Lemma \ref{inequivalent} proves
the inscribed maps related to different elements of $I(\zeta)$ are inequivalent,
hence $i: I(\zeta) \to E(f,\psi)$ is injective.
Proposition \ref{surjective} proves
an arbitrary extension $g: \Sigma_0 \to \Sigma$ of $f$ is an inscribed map related to an element $\{(\tilde{H}_1,H_1),\ldots,(\tilde{H}_n,H_n)\} \in I(\zeta)$,
hence $i: I(\zeta) \to E(f,\psi)$ is surjective.
We prove this 
by providing $g$ with a sequence of cancellation operations 
$(g,\Sigma_0) \stackrel{\tilde{H}_n}{\leadsto} (g_n,\Sigma_n) \stackrel{(\tilde{H}_{n-1},H_n)}{\leadsto} 
(g_{n-1},\Sigma_{n-1}) \stackrel{(\tilde{H}_{n-2},H_{n-1})}{\leadsto} \ldots
\stackrel{(\tilde{H}_{1},H_2)}{\leadsto} (g_1,\Sigma_1)$,
where $g_1$ is an embedding.

As a result,
we establsh $i: I(\zeta) \to E(f,\psi)$,
the bijection between $I(\zeta)$ and the set of equivalence classes of extensions of $f$ related to $\psi$:

\begin{thm}\label{main}
	For a closed oriented surface $\Sigma$,
	let $f: S^{1} \to \Sigma$ be a homologically trivial immersion 
	and $\psi$ a normal numbering of $f$.
	Fix $\zeta$ an inscribed set of $(f,\psi)$.
	Then there is a bijection between $I(\zeta)$ and the set of equivalence classes of extensions of $f$ related to $\psi$
\end{thm}

\subsection{Organization}\

Section \ref{section 2} gives basic ingredients.
Section \ref{section 3} introduces the cancellation operation $(g,\Sigma_0) \leadsto (g_1,\Sigma_1)$,
and Section \ref{section 4} provides the ways to yield a graph $G$ and construct the cancellable domains,
then determine the cancellation operation
$(g,\Sigma_0) \stackrel{G}{\leadsto} (g_1,\Sigma_1)$ (Case \ref{case 1}) or $(g,\Sigma_0) \stackrel{(G,g(G^{'}))}{\leadsto} (g_1,\Sigma_1)$ (Case \ref{case 2}).
Section \ref{section 5} defines inscribed maps,
and introduces the inscribed set $\zeta$ and $I(\zeta) \subseteq \zeta$.
Section \ref{section 6} proves the Theorem \ref{main}.
Section \ref{section 7} summaries of the general cases.
Section \ref{section 8} provides some examples.
	
	\section{Preliminaries}\ \label{section 2}

	This section is to introduce the basic ingredients:
	\emph{Numberings} (Definition \ref{numbering}) 
	and \emph{Gaussian circles} (Definition \ref{Gaussian circle}).

	\subsection{Numbering}\ \label{subsection 2.1}
	
	Fix a surface $\Sigma$ and an immersion $f: S^{1} \to \Sigma$.
	If there exists an extension $F: \Sigma_0 \to \Sigma$,
	then $f(S^{1})$ is homologically trivial.
	In \hyperref[McIntyre]{[11}, Section $1$]
	the numbering of preimages of points in $\Sigma \setminus f(S^{1})$ satisfies:
	at every segment of the immersed circle,
	the number to its left is $1$ greater than the number to its right.
	
	\begin{defn}[Numbering]\label{numbering}\rm
		For a closed oriented surface $\Sigma$,
		let $f: S^{1} \to \Sigma$ be a homologically trivial immersion.
		Assume $\{A_1,\ldots,A_n\}$ is the set of the components of $\Sigma \setminus f(S^{1})$.
		$\psi: \{A_1, \ldots, A_n\} \to \mathbb{Z}$ is a \emph{numbering} if:
		at every segment of $f(S^{1})$,
		the image of the component in its left is $1$ greater than the image of the component in its right.
		A numbering is \emph{normal} if its range is a subset of $\mathbb{Z}_{\geqslant 0}$.
	\end{defn}

The name ``normal numbering'' follows from \hyperref[Frisch]{[8}, Definition $1.2.1$].

\begin{lm}\label{lemma of numbering}
	For a closed oriented surface $\Sigma$,
	let $f: S^{1} \to \Sigma$ be a homologically trivial immersion.
	The (normal) numberings of $f$ exist uniquely up to choose its minimum image to be a (nonnegative) integer constant.
\end{lm}

\begin{proof}
	Basically follows from \hyperref[McIntyre, Cairns]{[10}, Lemma $2$].
\end{proof}

    Fix $F: \Sigma_0 \to \Sigma$ an extension of $f$.
	There exists a normal numbering $\psi_F: \{A_1,\ldots,A_n\} \to \mathbb{Z}_{\geqslant 0}$
	sending $A_i$ to the cardinality of $F^{-1}(x)$ ($\forall x \in A_i$), $\forall i \in \{1,2,\ldots,n\}$.
	We say $F$ is \emph{related to} $\psi_F$,
	and $\psi_F$ is the normal numbering \emph{given by} $F$.
	The extensions related to different normal numberings are inequivalent  (\hyperref[Frisch]{[8}, Lemma 5.3.2]).
	If $\Sigma$ is a $2$-sphere or a closed oriented surface of genus greater than one,
	\hyperref[McIntyre]{[11]} showed the maximum image of $\psi_F$ has an upper bound if $f$ a filling immersed curve,
	and solved Question \ref{question 3} (i) (the existence problem) in these cases.

More genreally,
we can extend the numbering to the cases of $\Sigma = \mathbb{R}^{2}$ or $\Sigma$ is a compact orientable surface with nonempty boundary.
Note that arbitrary extensions of $f$ (if exist) are related to a unique normal numbering $\psi_f$ in these cases (if $\Sigma = \mathbb{R}^{2}$, $\psi_f$ sends the components of $\mathbb{R}^{2} \setminus f(S^{1})$ to winding numbers).
Also,
we can extend the numbering to the case of a homologically trivial immersion $f: S^{1} \coprod \ldots \coprod S^{1} \to \Sigma$ (\hyperref[Frisch]{[8}, Lemma 1.2.5]).

The topology type of an immersed surface bounded by the immersed circle is determined by the normal numbering given by it (note that some changes need to be adopted for $3$-dimensional cases).
This basically follows from \hyperref[McIntyre]{[11}, Lemma $3$].
In the remainder of this paper,
we state with fixing a normal numbering of the immersed circle,
and we will not discuss the topology types specially.
	
	\subsection{Gaussian circles}\ \label{subsection 2.2}
	
	\begin{figure}\label{Milnor curve}
		\includegraphics[width=0.33\textwidth]{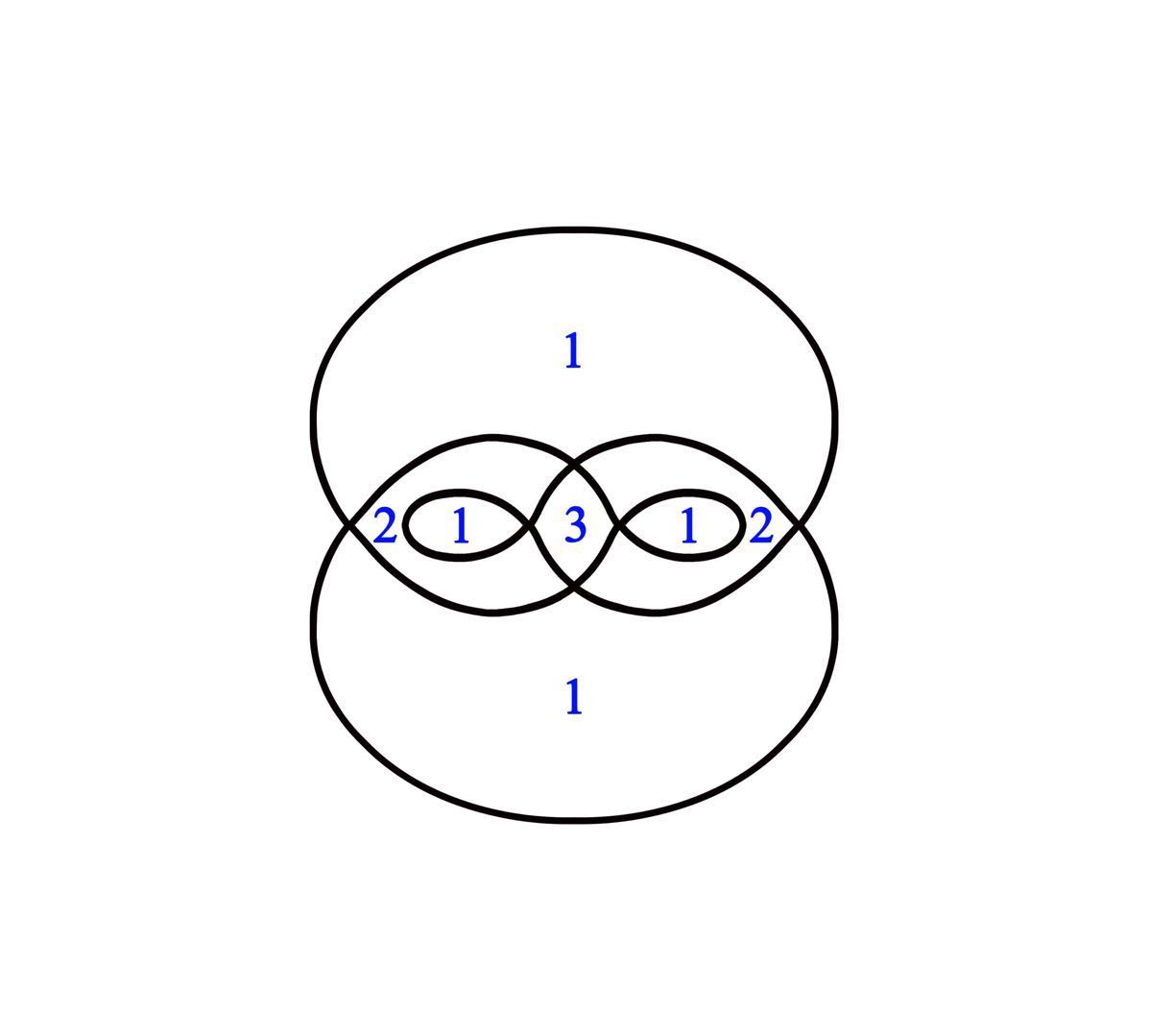}
		\includegraphics[width=0.33\textwidth]{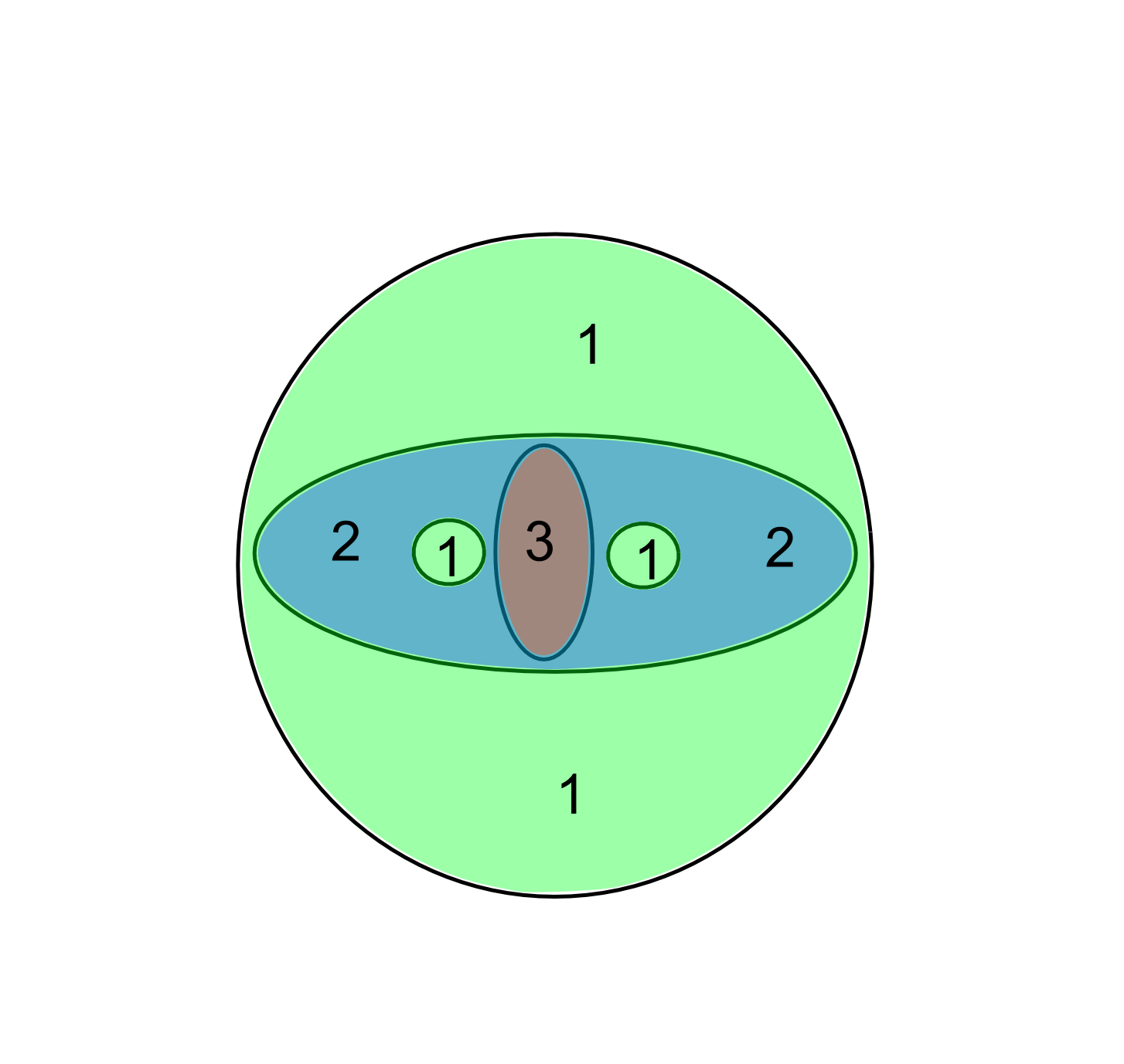}
		\includegraphics[width=0.33\textwidth]{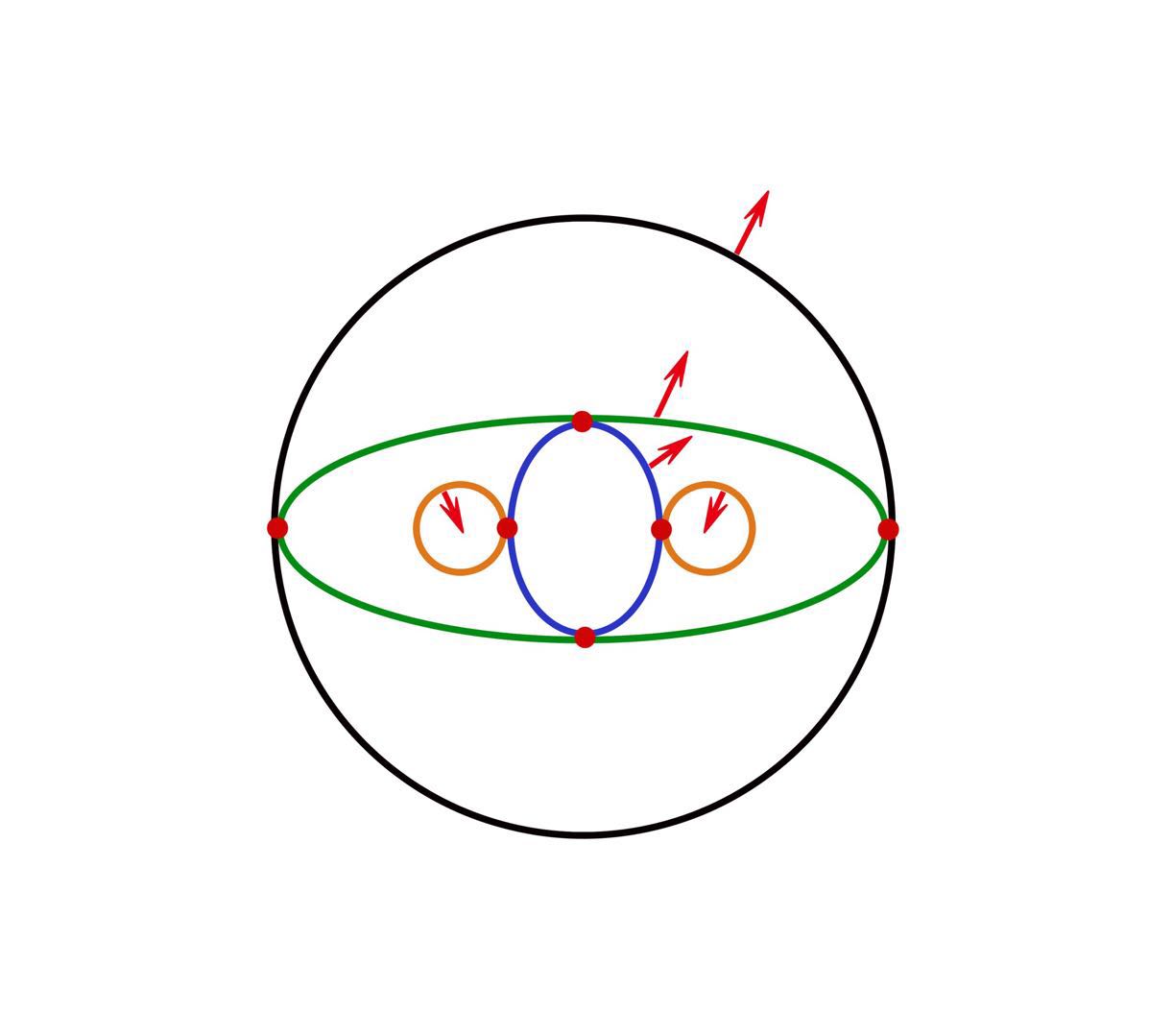}
		\caption{The Gaussian circles of the Milnor curve. (see \hyperref[Poenaru]{[12]})}
	\end{figure}

    \begin{defn}[Gaussian circles]\label{Gaussian circle}\rm
	For a closed oriented surface $\Sigma$,
	let $f: S^{1} \to \Sigma$ be a homologically trivial immersion
	and $\psi$ a normal numbering of $f$.
	Let $n = \max \psi$,
	$m = \min \psi$. 
	
	(i)
	Assume $\{A_1,\ldots,A_n\}$ is the set of the components of $\Sigma \setminus f(S^{1})$.
	Let $D_i(f,\psi)$ be the closure of $\bigcup_{k \in \{1,2,\ldots,n\}, \psi(A_k) \geqslant i} A_k$,
	$\forall i \in \{1,2,\ldots,n\}$.
	
	(ii)
	$\forall i \in \{m,m+1,\ldots,n\}$,
	each component of $\partial D_i(f,\psi)$ is an embedded circle.
	Call these circles the \emph{Gaussian circles} of $f$. 
	
	(iii)
	Let $V_i(f,\psi) = \partial D_i(f,\psi) \cap \partial D_{i-1}(f,\psi)$,
	$\forall m+1 \leqslant i \leqslant n$, 
	and $V_1(f,\psi) = \ldots = V_{m}(f,\psi) = \emptyset$. 
\end{defn}

A Gaussian circle is composed of segments of $f(S^{1})$ piecewise,
and $\bigcup_{i=m+1}^{n}V_i(f,\psi)$ is the set of nodes in $f(S^{1})$.
Similar to the numbering,
we can extend Definition \ref{Gaussian circle} to the case of a homologically trivial immersion $f: S^{1} \coprod \ldots \coprod S^{1} \to \Sigma$.

In \hyperref[McIntyre]{[11]},
$D_i(f,\psi)$ is denoted by $S_i$. The definition of Gaussian circles basically follows from \hyperref[Ezell, Marx]{[5}, Section $2$].
Note that our definition differs a little from \hyperref[Ezell, Marx]{[5}, Section $2$]:

    \begin{remark}\label{explanation of Gaussian circle}\rm
    	In \hyperref[Ezell, Marx]{[5]}, 
    	Gaussian circles are some disjoint embedded circles obtained by separating and smoothing the self-intersections of the immersed circle.
    	But we allow them to intersect at nodes.
    \end{remark}

\begin{figure}\label{cancellation picture 1}
	\centering 
	\includegraphics[width=0.45\textwidth]{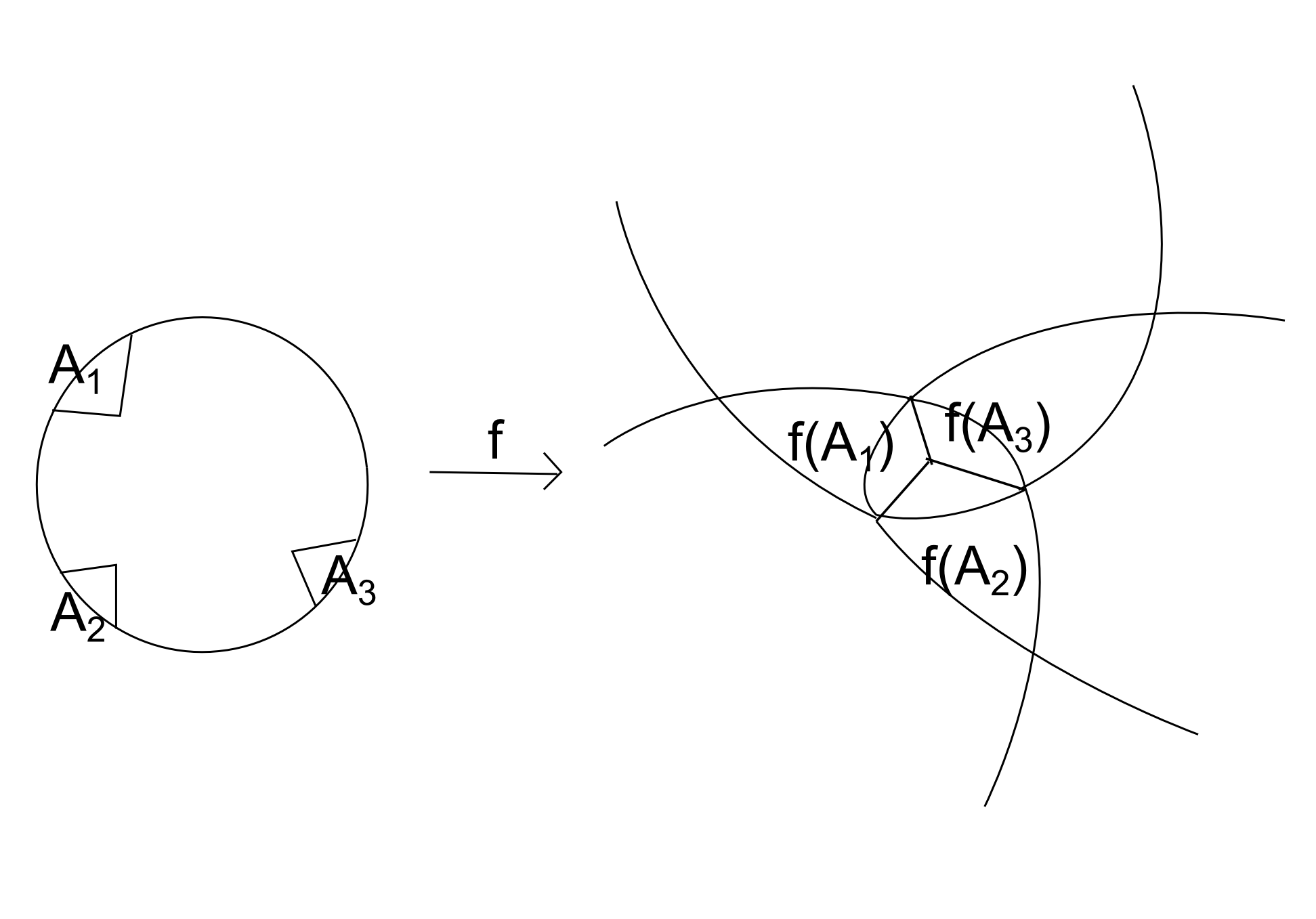}
	(a)
	\includegraphics[width=0.45\textwidth]{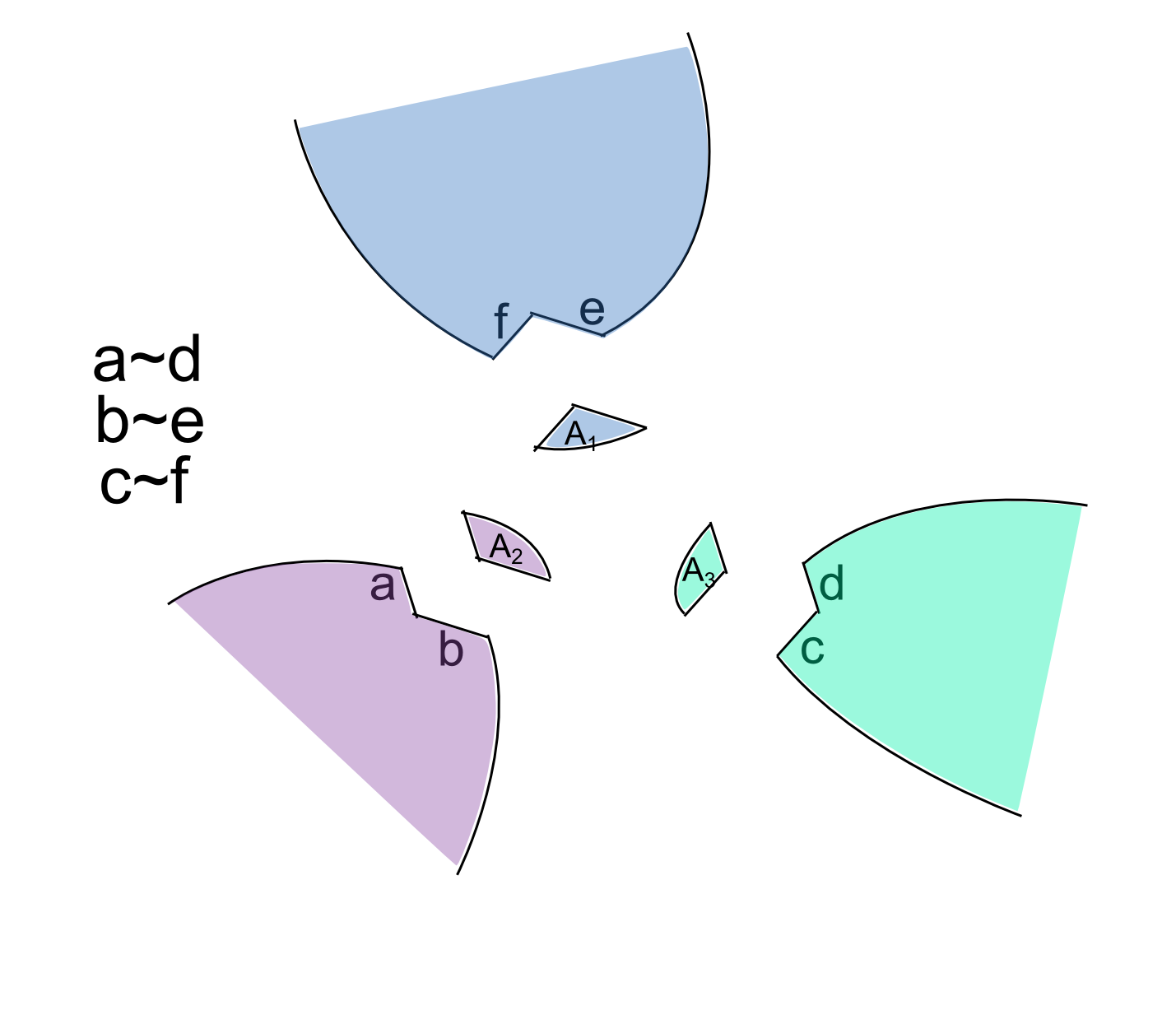}
	(b)
	\includegraphics[width=0.6\textwidth]{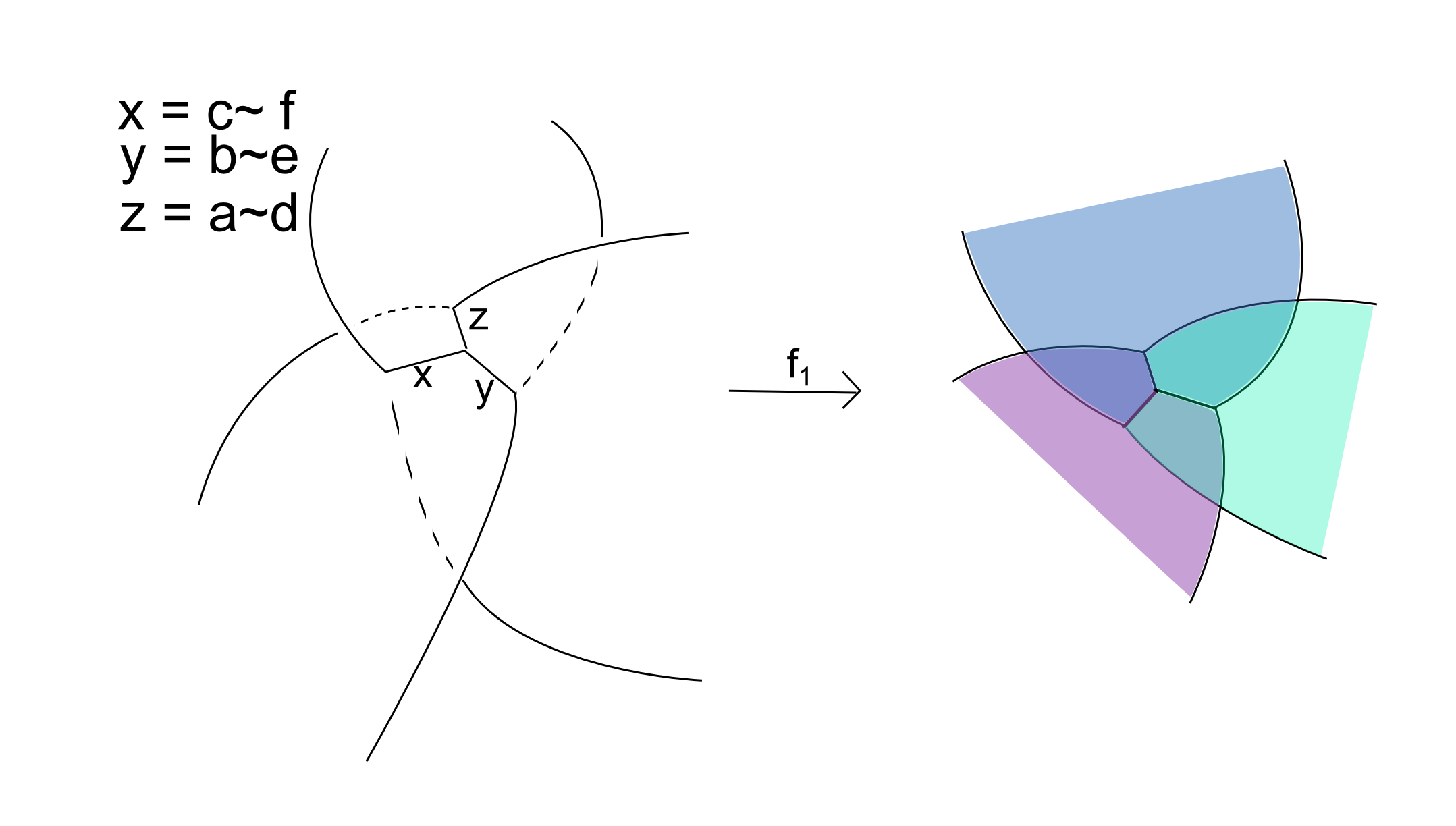}
	(c)
	\includegraphics[width=0.6\textwidth]{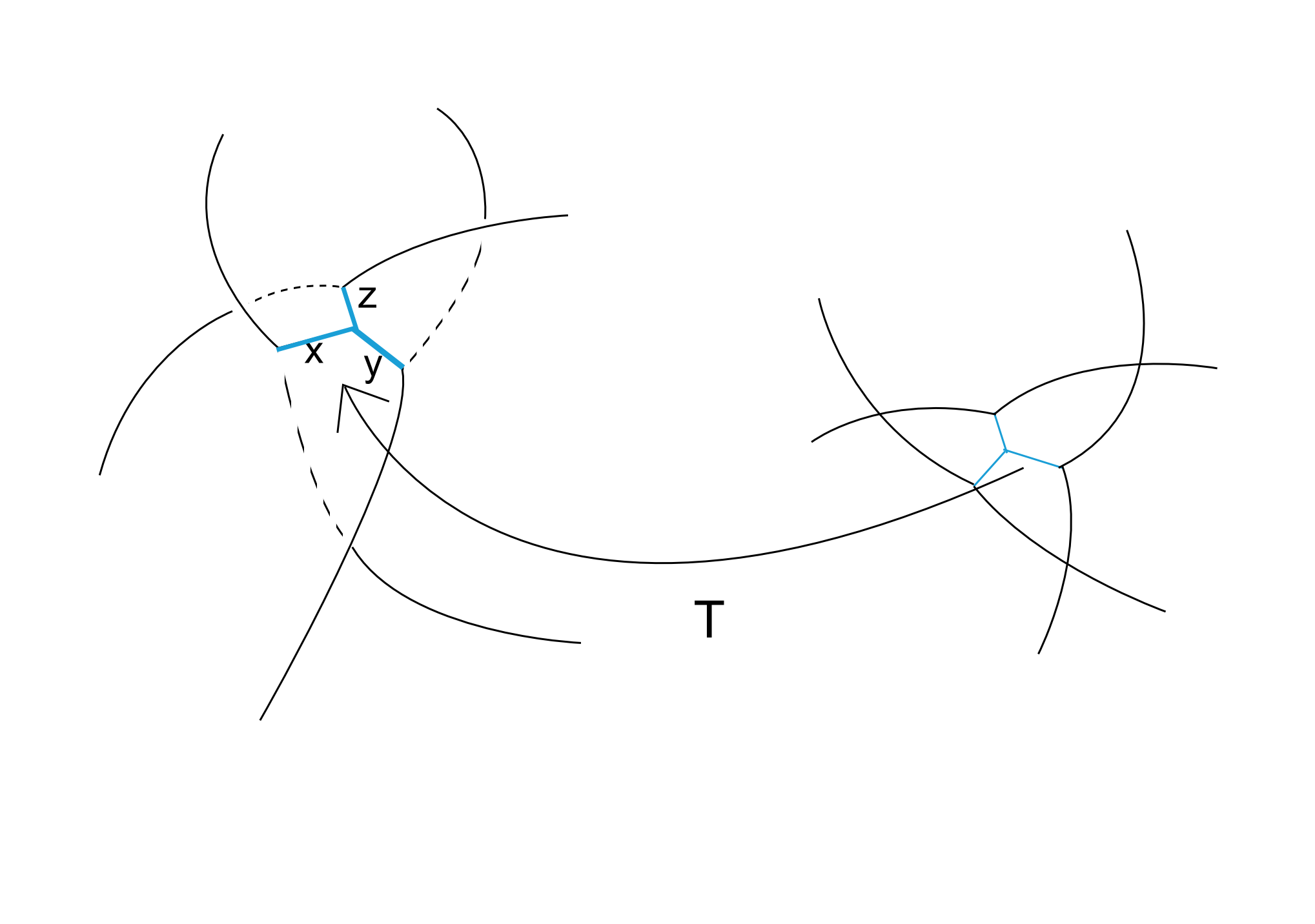}
	(d)
	\caption{the cancellation of $\{A_1,A_2,A_3\}$.}
\end{figure}
	
	\section{The cancellation operation}\ \label{section 3}
	
	This section is to
	introduce an operation to transform a polymersion (Definition \ref{polymersion}) of a surface into a surface.
	The intention is to simplify the polymersion.
	More precisely, given some \emph{cancellable domains} (Definition \ref{cancellable components}),
	the \emph{cancellation operation} (Definition \ref{cancellation operation} (ii)) transforms the polymersion by deleting their interior and identifying the segments with same images.
	This gives an \emph{associated map} (an embedding of a graph) if the cancellation is \emph{regular} (Definition \ref{regular}).
	
	Figure \ref{cancellation picture 1} provides an example of the cancellation operation.
	
	Definition \ref{polymersion} and Remark \ref{polymersion remark} basically follow from \hyperref[Ezell, Marx]{[5}, Section $1$].

\begin{defn}\rm\label{polymersion}
Let $M, N$ be compact orientable surfaces ($M$ may be disconnected).
$g: M \to N$ is a \emph{polymersion} if:

$\bullet$
At each $z \in \stackrel{\circ}{M}$ (where $\stackrel{\circ}{M}$ means the interior of $M$),
$g$ is topologically equivalent to $z \mapsto z^{k}$ ($k \in \mathbb{Z}_{\geqslant 1}$).

$\bullet$
If $\partial M \ne \emptyset$,
then there exists an open set $U$ such that
$\partial M \subseteq U$
and $g \mid_{U}$ is an immersion.
\end{defn}

Similar to immersions,
we will always assume a polymersion $g: M \to N$ such that the interior of $M$ is mapped to the left side of $g(\partial M)$ by $g$ and assume $g \mid_{\partial M}$ is a normal immersion,
if $\partial M \ne \emptyset$.

\begin{remark}\rm\label{polymersion remark}
	For a polymersion $g: M \to N$, 
	a point $z \in \stackrel{\circ}{M}$ of local form $z \mapsto z^{n}$ ($n \geqslant 2$) is called a \emph{critical point} of \emph{multiplicity} $n-1$,
	and its image is called a \emph{branch point}.
	
	In this paper, 
	if $M, N$ are compact orientable surfaces ($M$ may be disconnected) and $g: M \to N$ is a polymersion,
	we will always assume
	there is no branch point in $g(\partial M)$,
	and assume different critical points of $g$ have different images if not otherwise mentioned.
	The \emph{index} of a branch point is the multiplicity of the critical point mapped to it.
\end{remark}

Recall that Definition \ref{Gaussian circle} can be extended to the case of a homologically trivial immersion $f: S^{1} \coprod \ldots \coprod S^{1} \to \Sigma$,
and $D_{\max \psi}(g(\partial \Sigma_0), \psi)$ is independent of the normal numbering $\psi$.

\begin{defn}\label{R(g)}\rm
For a closed oriented surface $\Sigma$
and a compact orientable surface $\Sigma_0$ ($\Sigma_0$ may be disconnected),
let $g: \Sigma_0 \to \Sigma$ be a polymersion.
Let $R(g)$ be $D_{\max \psi}(g(\partial \Sigma_0), \psi)$ ($\psi$ is an arbitrary normal numbering) if $\partial \Sigma_0 \ne \emptyset$,
and $R(g) = \Sigma$ if $\partial \Sigma_0 = \emptyset$.
\end{defn}

Note that $R(g)$ lies in the left of each segment in $\partial R(g) \subseteq g(\partial \Sigma_0)$.

\begin{defn}[Cancellable domains]\label{cancellable components}\rm
For a closed oriented surface $\Sigma$
and a compact orientable surface $\Sigma_0$ ($\Sigma_0$ may be disconnected),
let $g: \Sigma_0 \to \Sigma$ be a polymersion.
Assume $A_1, A_2, \ldots, A_n \subseteq \Sigma_0$ are closed domains such that
$\stackrel{\circ}{A_1}, \stackrel{\circ}{A_2}, \ldots, \stackrel{\circ}{A_n}$ are homeomorphically embedded into $R(g)$ by $g$.
$A_1, A_2, \ldots, A_n$ are called \emph{cancellable} if:

$\bullet$
There exists an embedded graph $G \subseteq R(g)$ such that $G \cap \partial R(g) = \{v \mid v \in V(G), deg_G(v) = 1\}$,
and
$\{g(\stackrel{\circ}{A_1}), g(\stackrel{\circ}{A_2}), \ldots, g(\stackrel{\circ}{A_n})\}$ 
is the set of the components of $R(g) \setminus G$.
($G$ is called the graph \emph{associated} to $A_1, A_2, \ldots, A_n$)

$\bullet$
$(g \mid_{A_i})^{-1}(g(A_i)\cap \partial R(g)) \subseteq \partial \Sigma_0$ if $g(A_i) \cap \partial R(g) \ne \emptyset$,
$\forall i \in \{1,2,\ldots,n\}$.
\end{defn}

\begin{defn}[Cancellation operation]\label{cancellation operation}\rm
	For a closed oriented surface $\Sigma$
	and a compact orientable surface $\Sigma_0$ ($\Sigma_0$ may be disconnected),
	let $g: \Sigma_0 \to \Sigma$ be a polymersion.
	Assume that the closed domains $A_1, A_2, \ldots, A_n \subseteq \Sigma_0$ are cancellable.
	The \emph{cancellation of $\{A_1,A_2,\ldots,A_n\}$} (or, \emph{canceling $\{A_1,A_2,\ldots,A_n\}$})
	$(g,\Sigma_0) \leadsto (g_1,\Sigma_1)$
	is the following procedure:
	
	$\bullet$
	Let $\Sigma^{'}_{0}$ be $\overline{\Sigma_0 \setminus (A_1 \cup A_2 \cup \ldots \cup A_n)}$.
	$g_0 = g \mid_{\Sigma^{'}_{0}}$.
	Let $h$ be the equivalence relation such that 
	$x \stackrel{h}{\sim} y$ if
	$x,y \in \partial \Sigma_{0}^{'} \cap g_{0}^{-1}(G), g_0(x) = g_0(y)$.
	Let $\Sigma_{1}$ be the identification space $\Sigma_{0}^{'} / \sim_h$.
	Let $h_*: \Sigma_{0}^{'} \to \Sigma_1$ be the identification map induced by $h$.
	Let $g_1: \Sigma_1 \to \Sigma$ be the map given by following commutative diagram.
	\begin{center}
		$\xymatrix{
			& \Sigma^{'}_0 \ar[d]^{h_*} \ar[r]_{g_0}
			& \Sigma \ar[d]_{id}       \\
			& \Sigma_1 \ar[r]^{g_1}   & \Sigma                }$
	\end{center}
Then $(g,\Sigma_0) \leadsto (g_1,\Sigma_1)$ has defined.
The identification map $h_*$ is called the \emph{cancellation map} of $A_1,A_2,\ldots,A_n$.
\end{defn}

The cancellation $(g,\Sigma_0) \leadsto (g_1,\Sigma_1)$ is depend on the choice of cancellable domains $A_1,A_2,\ldots,A_n$.
In Section \ref{section 4},
we determine the cancellation by a graph $G$ or a pair of graphs $(G,g(G^{'}))$,
and denote it by $(g,\Sigma_0) \stackrel{G}{\leadsto} (g_1,\Sigma_1)$ 
or $(g,\Sigma_0) \stackrel{(G,g(G^{'}))}{\leadsto} (g_1,\Sigma_1)$.

\begin{lm}\label{cancellation-well}
	The identification space $\Sigma_1$ is a compact orientable surface that may be disconnected,
	and the map $g_1: \Sigma_1 \to \Sigma$ is a polymersion.
\end{lm}

For each $s \in E(G)$,
there are $2$ arcs in $\coprod_{i=1}^{n} \partial A_i$ mapped homeomorphically to $s$ by $g$,
and they may be the same arc in $\Sigma_0$.
$g(\partial(A_1 \cup A_2 \cup \ldots \cup A_n)) \cap G$ is the subgraph of $G$ consisting of 
all $s \in E(G)$ that have $2$ different preimages in $A_1 \cup A_2 \cup \ldots \cup A_n$ (Figure \ref{G(...)}).

\begin{defn}\label{regular}\rm
For a closed oriented surface $\Sigma$
and a compact orientable surface $\Sigma_0$ ($\Sigma_0$ may be disconnected),
let $g: \Sigma_0 \to \Sigma$ be a polymersion.
Assume that the closed domains $A_1, A_2, \ldots, A_n \subseteq \Sigma_0$ are cancellable.

(i)
Let $G(A_1,A_2,\ldots,A_n) = g(\partial(A_1 \cup A_2 \cup \ldots \cup A_n)) \cap G$.

(ii)
Let $h_*: \Sigma^{'}_{0} \to \Sigma_1$ be the cancellation map of $A_1, A_2, \ldots, A_n$ ($\Sigma^{'}_{0} = \overline{\Sigma_0 \setminus (A_1 \cup A_2 \cup \ldots \cup A_n)}$).
The cancellation of $A_1, A_2, \ldots, A_n$ is called \emph{regular}
if $\#(h_*(\partial \Sigma_{0}^{'} \cap g_{*}^{-1}(x))) = 1$,
$\forall x \in V(G(A_1,A_2,\ldots,A_n))$.
\end{defn}

If the cancellation of $A_1, A_2, \ldots, A_n$ is regular,
then $\#(h_*(\partial \Sigma_{0}^{'} \cap g_{*}^{-1}(x))) = 1, \forall x \in G(A_1,A_2,\ldots,A_n)$.
So there is a map $T: G(A_1,A_2,\ldots,A_n) \to \Sigma_1$ such that
$\{T(x)\} = h_*(\partial \Sigma_{0}^{'} \cap g_{*}^{-1}(x))$,
$\forall x \in G(A_1,A_2,\ldots,A_n)$.
Note that the image of $G(A_1,A_2,\ldots,A_n) \cap \partial R(g)$ under $T$ lies in $\partial \Sigma_1$.
Call $T$ the \emph{associated map} of canceling $\{A_1, A_2, \ldots, A_n\}$.

\begin{figure}\label{G(...)}
	\centering 
	\includegraphics[width=0.5\textwidth]{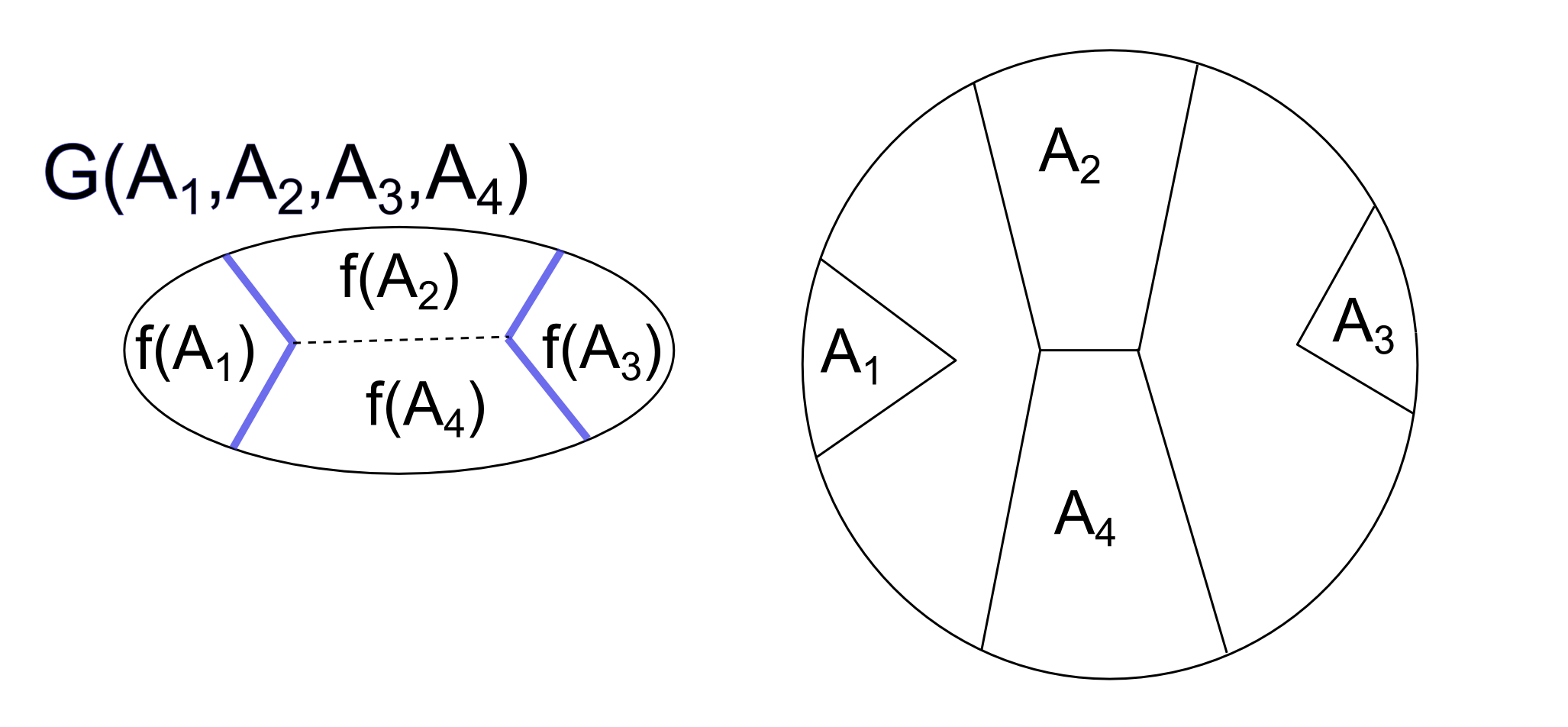}
	\caption{$G(A_1,A_2,A_3,A_4)$.}
\end{figure}
	
	\section{The construction of cancellations}\ \label{section 4}
	
	In last section,
	we define the cancellation operation,
	which depends on the choice of cancellable domains (Definition \ref{cancellable components}).
	Given $g: \Sigma_0 \to \Sigma$ a polymersion,
	this section is mainly concerned with constructing cancellable domains.
	More precisely,
	$g$ yields an embedded graph $G \subseteq R(g)$,
	and $G$ determine the cancellable domains in $\Sigma_0$.
	Then the cancellation operation $(g,\Sigma_0) \stackrel{G}{\leadsto} (g_1,\Sigma_1)$ (in Case \ref{case 1}) or $(g,\Sigma_0) \stackrel{(G,g(G^{'}))}{\leadsto} (g_1,\Sigma_1)$ (in Case \ref{case 2}, where $G^{'} \subseteq  \Sigma_0$ is given in Case \ref{case 2}) is constructed.
	
	Subsection \ref{subsection 4.1} gives the way to yield an embedded graph in a union of closed regions.
	Subsection \ref{subsection 4.2} introduces the \emph{construction triple} (Definition \ref{construction triple}) to construct cancellable domains,
	and Subsection \ref{subsection 4.3} constructs cancellable domains by yielding an embedded graph and choosing a construction triple.
	
\subsection{trivalent graphs}\ \label{subsection 4.1}

We define some embedded graphs generated in a union of closed regions.
Such graphs are used to construct cancellable domains in Subsection \ref{subsection 4.3}.

\begin{defn}\label{case 1 graph}\rm 
	Let $A$ be a union of compact orientable surfaces with nonempty boundaries.
	Let $P \subseteq \partial A$ be a finite set of points 
	(may be $\emptyset$).  
	
	(i) 
	$(A,P)$ is \emph{appropriate} if:
	$A$ has no component $A_0$ 
	such that 
	$A_0$ is a disk 
	and $\#(A_0 \cap P) = 1$. 
	
	(ii)
	Assume $(A,P)$ is appropriate.
	An embedded graph $G \subseteq A$ is an \emph{$(A,P)$-trivalent graph} if:
	Each vertex of $G$ has degree no more than $3$,
	$P = \{v \mid v \in V(G), deg_G(v) = 1\}$,
	and the following holds for each component $A_0$ of $A$:
	
	$\bullet$
	If $A_0$ is not a disk or $\#(A_0 \cap P) \geqslant 2$,
	then
	$\bigcup_{e \in E(G), e \subseteq A_0, e \cap P = \emptyset} e$ is a deformation retract of $A_0$.
	
	$\bullet$
	If $A_0$ is a disk and $A_0 \cap P = \emptyset$,
	then $A_0 \cap G = \emptyset$.
\end{defn}

\begin{figure}\label{case 1 picture}
	\includegraphics[width=0.33\textwidth]{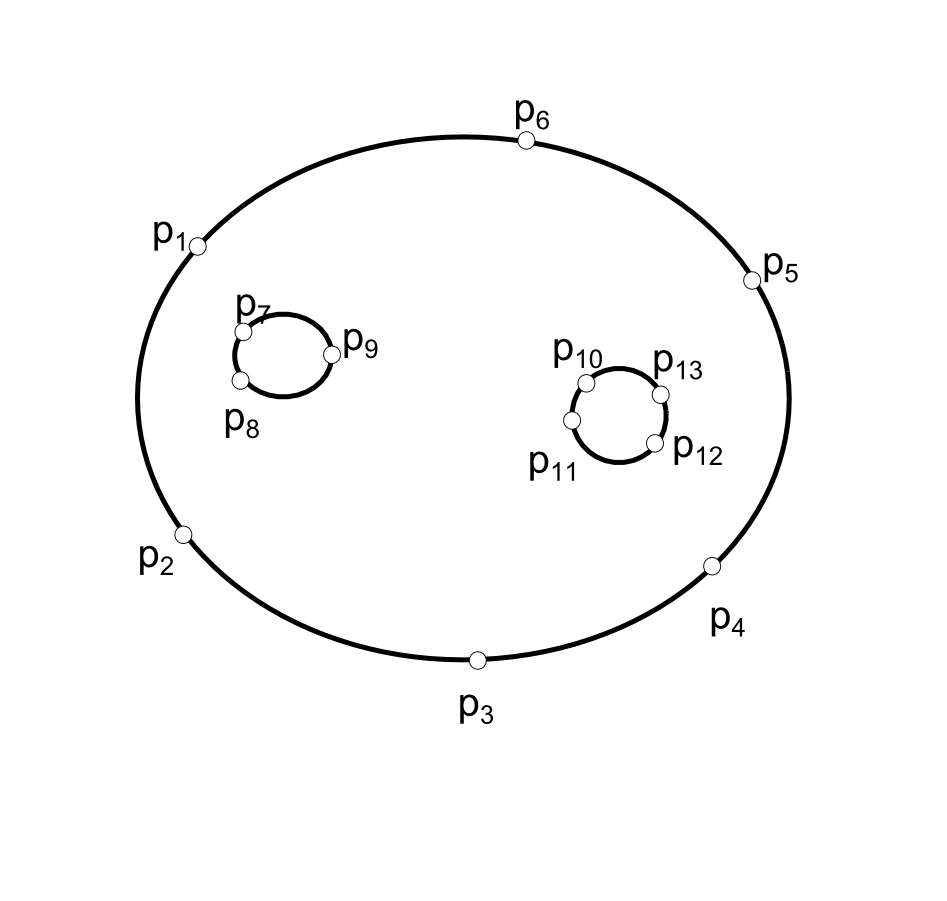}
	\includegraphics[width=0.33\textwidth]{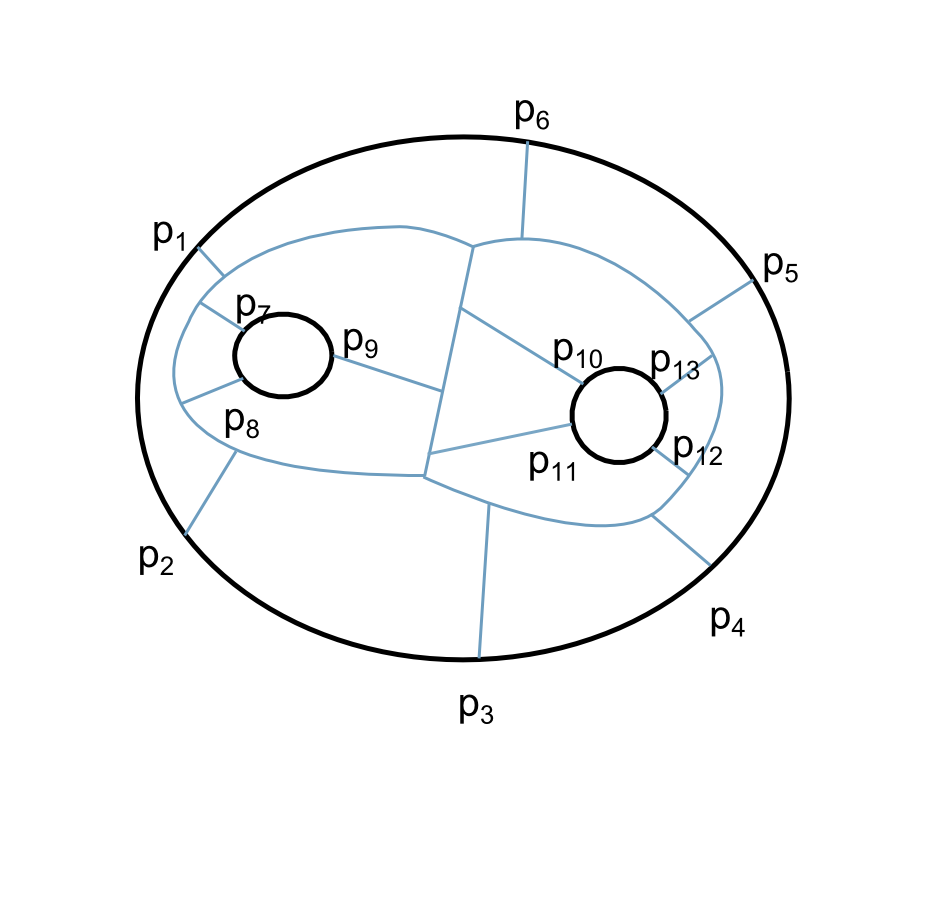}
	\includegraphics[width=0.33\textwidth]{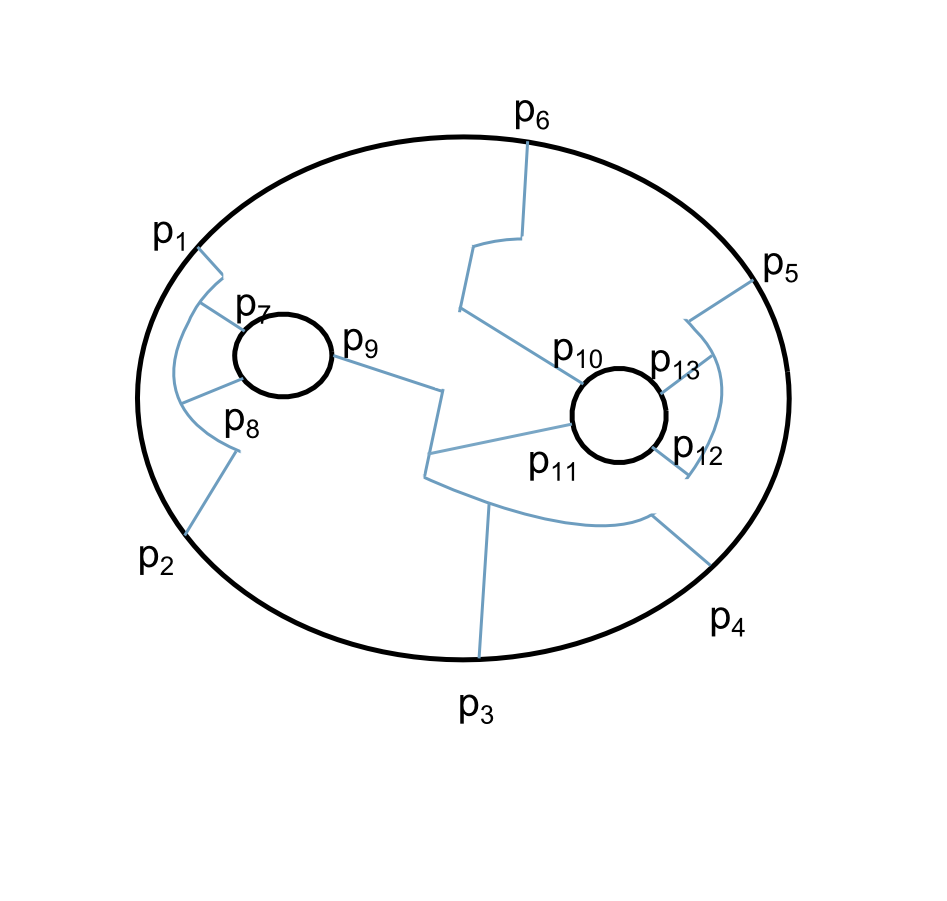}
	\caption{the $(A,P)$-trivalent graph, and a good subgraph of it.}
\end{figure}	

\begin{defn}\label{good subgraph}\rm
	Let $X$ be a graph whose vertices have degree no more than $3$.
	Assume $Y$ is a subgraph of $X$ such that arbitrary vertices of $Y$ have degree no less than $1$.
	$Y$ is called a \emph{good subgraph} of $X$ if
	$\{v \mid v \in V(X), deg_X(v) = 1\} = \{v \mid v \in V(Y), deg_Y(v) = 1\}$.
	
	We denote the set of all good subgraphs of $X$ by $sub(X)$.
\end{defn}

\begin{defn}\label{case 2 graph}\rm
	Let $\Sigma$ be a closed oriented surface,
	$A \subseteq \Sigma$ is a union of closed regions ($A$ may be $\Sigma$).
	Let $G_0 \subseteq A$ be an embedded graph such that
	the vertices have degree no more than $3$,
	and $G_0 \cap \partial A = \{v \mid v \in V(G_0), deg_{G_0}(v) = 1\}$.
	Let $P \subseteq (\partial A \setminus G_0) \cup \{v \mid v \in V(G_0), deg_{G_0}(v) = 3\}$ be a finite set of points.
	
	(i)
	Let $a_1$, $a_2$, \ldots, $a_{m}$
	be the components obtained by
	cutting off $G_0$ from $A$ (which means to delete $G_0$ from $A$ and do a path compactification).
	For each $k \in \{1,2,\ldots,m\}$,
	let $i_k: a_k \to A$ be the continuous map such that 
	the restriction of $i_k$ to $\stackrel{\circ}{a_k}$ is an inclusion,
	and let $P_k = \{x \mid x \in \partial a_k, i_k(x) \in P\}$.
	$(A,G_0,P)$ is called \emph{appropriate} if
	$(a_k,P_k)$ is appropriate,
	$\forall k \in \{1,2,\ldots,m\}$.
	
	(ii)
	Assume $(A,G_0,P)$ is appropriate.
	$G \subseteq A$ is called a \emph{thin $(A,G_0,P)$-trivalent graph} if:
	there exists $G_k \subseteq a_k$ an $(a_k,P_k)$-trivalent graph
	for each $k \in \{1,2,\ldots,m\}$,
	such that $G = \bigcup_{k=1}^{m} i_k(G_k)$.
	
	(iii)
	Let $H$ be a subgraph of $G$.
	We cut off $G_0 \cap G$ from $G$, and obtain $G^{'}$.
	$H$ is a \emph{$G_0$-good subgraph} of $G$ if:
	for each component $L$ of $G^{'}$,
	$H \cap L$ is a good subgraph of $L$.
	We denote by $sub_{G_0}(G)$ the set of $G_0$-good subgraphs of $G$.
\end{defn}

\begin{figure}\label{case 2 picture}
	\includegraphics[width=0.33\textwidth]{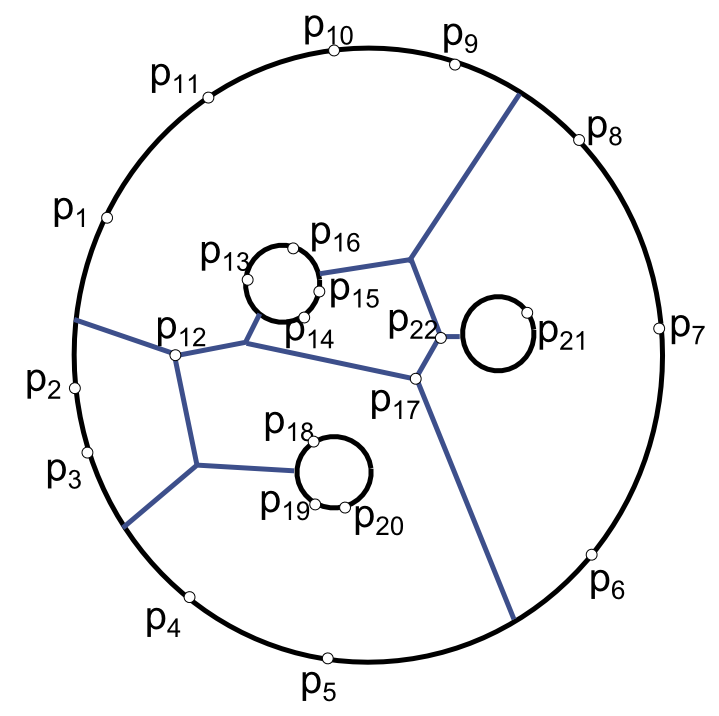}
	\includegraphics[width=0.33\textwidth]{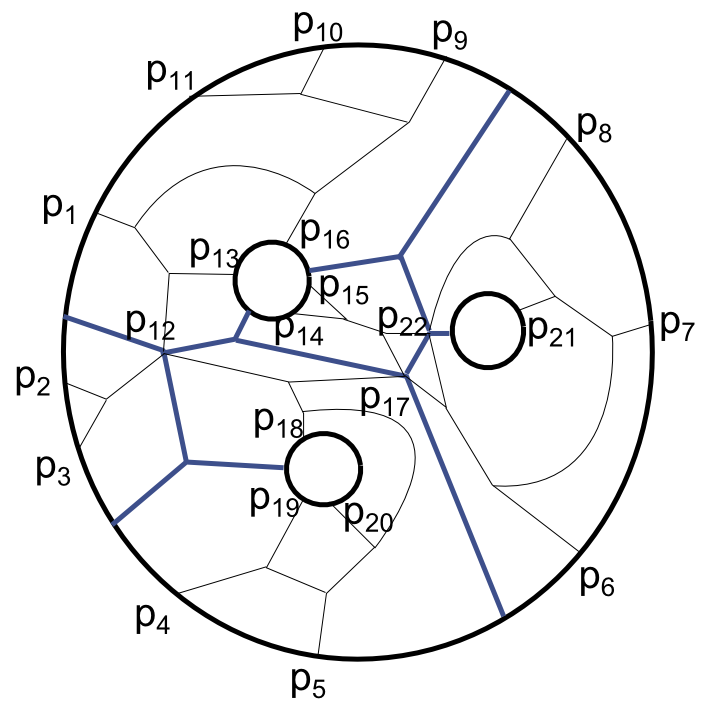}
	\includegraphics[width=0.33\textwidth]{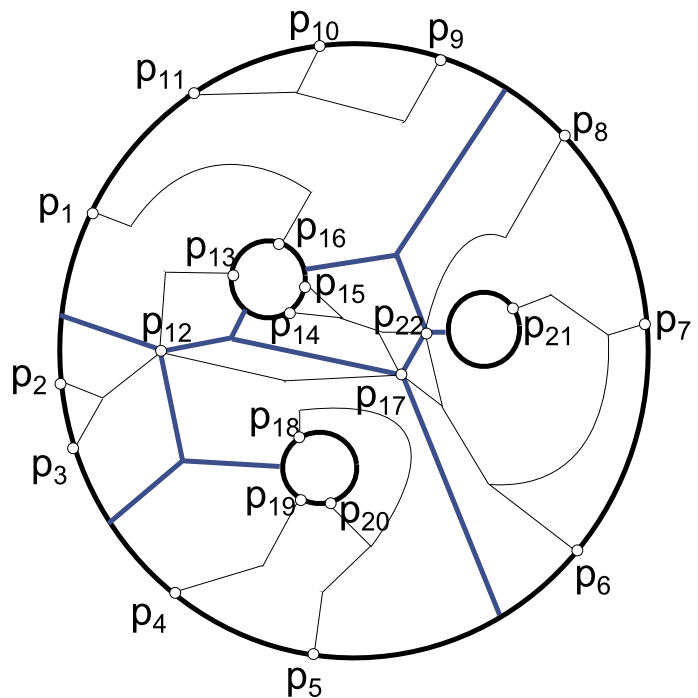}
	\caption{the thin $(A,G_0,P)$-trivalent graph, and a $G_0$-good subgraph of it.}
\end{figure}

For our convenience,
we give the following generalizations:

\begin{remark}\label{remark of emptyset}\rm
	We include the case $G_0 = \emptyset$ if $A \ne \Sigma$.
	In this case,
	$P \subseteq \partial A$,
	$(A,G_0,P)$ is appropriate if and only if $(A,P)$ is appropriate,
	and a thin $(A,G_0,P)$-trivalent graph is a $(A,P)$-trivalent graph.
	Moreover,
	$sub_{G_0}(G) = sub(G)$.
\end{remark}
	
\subsection{Construction triple}\ \label{subsection 4.2}

We consider a triple to construct cancellable domains.
Recall that $R(g) = D_{\max \psi} (g(\partial \Sigma_0),\psi)$ for an arbitrary normal numbering $\psi$ (Definition \ref{R(g)}).

\begin{defn}\rm\label{construction triple}
	For a closed oriented surface $\Sigma$
	and a compact orientable surface $\Sigma_0$ ($\Sigma_0$ may disconnected),
	let $g: \Sigma_0 \to \Sigma$ be a polymersion.
	
	(i)
	The \emph{construction triple} $(\{p_1,p_2,\ldots,p_n\},G,\{\tilde{p}_1,\tilde{p}_2,\ldots,\tilde{p}_n\})$ is given by:
	
	$\bullet$
	$p_1,p_2,\ldots,p_n \in R(g)$ are some points.
	
	$\bullet$
	$G \subseteq R(g)$ is a graph such that $G \cap \partial R(g)$ is the set of leaves in $G$,
	and each component of $R(g) \setminus G$ includes one element of $\{p_1,p_2,\ldots,p_n\}$ exactly.
	
	$\bullet$
	$\tilde{p}_1,\tilde{p}_2,\ldots,\tilde{p}_n \in \Sigma_0$,
	$\tilde{p}_k \in g^{-1}(p_k), \forall k \in \{1,2,\ldots,n\}$.
	
	(ii)
	A construction triple 
	$(\{p_1,p_2,\ldots,p_n\},G,\{\tilde{p}_1,\tilde{p}_2,\ldots,\tilde{p}_n\})$ is \emph{suitable} if there exist 
	closed domains $A_1,A_2,\ldots,A_n \subseteq \Sigma_0$ such that:
	
	$\bullet$
	$\stackrel{\circ}{A_1}, \stackrel{\circ}{A_2}, \ldots, \stackrel{\circ}{A_n}$ are
	homeomorphically embedded into $\Sigma$ by $g$,
	and $\{g(\stackrel{\circ}{A_1}), g(\stackrel{\circ}{A_2}), \ldots, g(\stackrel{\circ}{A_n})\}$ 
	is the set of the components of $R(g) \setminus G$.
	
	$\bullet$
	$\tilde{p}_k \in A_k$,
	$\forall k \in \{1,2,\ldots,n\}$.
	
	$\bullet$
	$A_1,A_2,\ldots,A_n$ are cancellable.
\end{defn}

Moreover,
in the case of (ii),
$A_1,A_2,\ldots,A_n$ are said to be the cancellable domains \emph{given by} $(\{p_1,p_2,\ldots,p_n\},G,\{\tilde{p}_1,\tilde{p}_2,$ $\ldots,\tilde{p}_n\})$.

\subsection{The construction}\ \label{subsection 4.3}

Let $\Sigma_0$ be a compact orientable surface ($\Sigma_0$ may disconnected) and $\Sigma$ a closed oriented surface.
Let $g: \Sigma_0 \to \Sigma$ be a polymersion.

This subsection is to construct cancellable domains in following two cases.

\begin{case}\rm\label{case 1}
	$\partial \Sigma_0 \ne \emptyset$ (then $\partial R(g) \ne \emptyset$),
	and there is no branch point in $R(g)$. 
\end{case}

Let $N$ be the set of nodes of $g(\partial \Sigma_0)$ in $\partial R(g) \subseteq g(\partial \Sigma_0)$.

\begin{lm}\label{case 1 appropriate}
	$(R(g),N)$ is appropriate.
\end{lm}

\begin{proof}
	If not,
	then there exists a component $D_0 \subseteq R(g)$,
	$D_0$ is a disk and $\#(N \cap D_0) = 1$.
	Assume without loss of generality that $N \cap D_0 = \{p\}$.
	Let $D^{'}_{0}$ be the component of $g^{-1}(D_0)$ such that $(g \mid_{\partial \Sigma_0})^{-1}(\partial D \setminus \{p\}) \subseteq D^{'}_{0}$,
	then $D^{'}_{0}$ includes $2$ different preimages of $p$.
	This contradicts to $D_0$ is a disk ($D^{'}_{0}$ must be mapped homeomorphically to $D_0$ by $g$).
	So $(R(g),N)$ is appropriate.
\end{proof}

\begin{lm}\label{case 1 cancellation}
Assume $p_1,p_2,\ldots,p_n \in \partial R(g) \setminus N$ are $n$ points such that
each component of $\partial R(g) \setminus N$ includes one of them exactly.
$\{\tilde{p}_k\} = g^{-1}(p_k) \cap \partial \Sigma_0, \forall k \in \{1,2,\ldots,n\}$.
Assume $G$ is an arbitrary $(R(g),N)$-trivalent graph.
Then the construction triple $(\{p_1,p_2,\ldots,p_n\},G,\{\tilde{p}_1,\tilde{p}_2,\ldots,\tilde{p}_n\})$
is suitable,
and the cancellable domains given by it are independent of the choice of $p_1,p_2,\ldots,p_n,\tilde{p}_1,\tilde{p}_2,\ldots,\tilde{p}_n$.
\end{lm}

\begin{proof}
	Fix $k \in \{1,2,\ldots,n\}$.
	Let $B_k$ be the component of $R(g) \setminus G$ containing $p_k$.
	Let $\tilde{B}_k$ be the component of $g^{-1}(B_k)$ containing $\tilde{p}_k$.
	
	First,
	we prove that $\tilde{B}_k$ is mapped homeomorphically to $B_k$ by $g$.
	Assume $R_0$ is the component of $R(g)$ containing $B_k$.
	If $R_0$ is a disk and $N \cap R_0 = \emptyset$,
	then $B_k = R_0$,
	and each component of $g^{-1}(B_k)$ is mapped homeomorphically to $B_k$ by $g$.
	If $R_0$ is not a disk or $N \cap R_0 \ne \emptyset$,
	assume $l_k$ is the component of $\partial R(g) \setminus N$ such that $l_k \subseteq B_k$,
	then $B_k \backsimeq l_k$.
	Let $\tilde{l}_k = g^{-1}(l_k) \cap \partial \Sigma_0$,
	then $\tilde{l}_k \subseteq \tilde{B}_k$, 
	and $\tilde{l}_k$ is mapped homeomorphically to $l_k$ by $g$.
	Hence $\tilde{B}_k$ is mapped homeomorphically to $B_k$ by $g$.
	
	Easily,
	$\tilde{B}_1,\ldots,\tilde{B}_n$ are cancellable,
	and they are independent of the choice of $p_1,p_2,\ldots,p_n,\tilde{p}_1,\tilde{p}_2,$ $\ldots,\tilde{p}_n$.
\end{proof}

With above conditions,
the cancellable domains given by  $(\{p_1,p_2,\ldots,p_n\},G,\{\tilde{p}_1,\tilde{p}_2,\ldots,\tilde{p}_n\})$
(which are determined by $G$)
are called the \emph{$G$-cancellable domains}.
Hence $G$ defines a cancellation $(g,\Sigma_0) \stackrel{G}{\leadsto} (g_1,\Sigma_1)$
(the cancellation is determined uniquely by $G$).

\begin{lm}\label{case 1 subgraph}
	Assume $A_1,A_2,\ldots,A_n$ are the $G$-cancellable domains, then $G(A_1,A_2,\ldots,A_n) \in sub(G)$ (where $G(A_1,A_2,\ldots,A_n)$ is defined in Definition \ref{regular} (i)).
\end{lm}

Next,
we consider the other case.

\begin{case}\rm\label{case 2}
	In this case, 
	$\partial \Sigma_0$ may be $\emptyset$ ($R(g) = \Sigma$ if $\partial \Sigma_0 = \emptyset$).
	There is an embedded graph $G^{'} \subseteq \Sigma_0$ such that (i) $\sim$ (v) hold:
	
	(i)
	The vertices of $G^{'}$ have degree no more than $3$.
	
	(ii)
	$G^{'}$ is homeomorphically embedded into $R(g)$ by $g$.
	
	(iii)
	The leaves of $G^{'}$ are in $\partial \Sigma_0$.
	
	(iv)
	Each critical point in $\Sigma_0$ mapped into $R(g)$ is a vertex of $G^{'}$ with degree $3$.
	 
	(v)
	For each branch point $y \in R(g)$,
	assume $a,b,c$ are the $3$ edges of $g(G^{'})$ at $y$ clockwise.
	Let $x \in \Sigma_0$ such that $\{x\} = g^{-1}(y) \cap G^{'}$ (then $x$ is a critical point)
	and $a^{'}, b^{'}, c^{'}$ the $3$ edges of $G^{'}$ at $x$ mapped to $a,b,c$.
	Then they are in the order $a^{'}, c^{'}, b^{'}$ clockwise around $x$.
\end{case}

\begin{figure}\label{case 2-picture}
	\centering
	\includegraphics[width=0.6\textwidth]{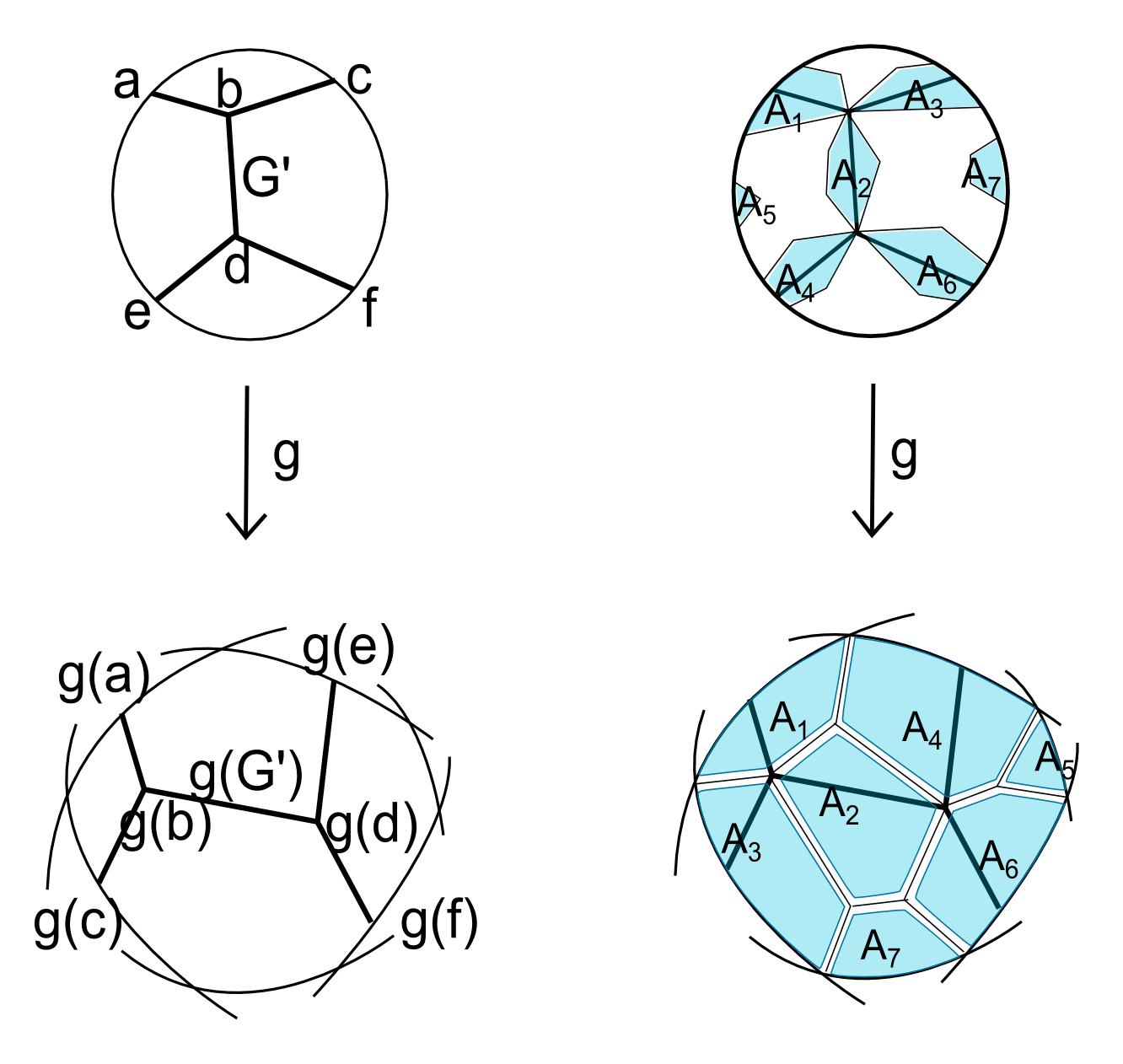}
	\caption{the cancellable domains in case $2$ (where $b$, $d$ are critical points).}
\end{figure}

Let $N$ be the union of the branch points in $D$ and the nodes of $g(\partial \Sigma_0)$ in $\partial D_g \subseteq g(\Sigma_0)$.

\begin{lm}\label{case 2 appropriate}
	$(R(g),g(G^{'}),N)$ is appropriate.
\end{lm}

\begin{proof}
	See the proof of Lemma \ref{case 1 appropriate}.
\end{proof}

\begin{lm}\label{case 2 cancellation}
	Assume $p_1,p_2,\ldots,p_n \in (g(G^{'}) \cup \partial R(g)) \setminus N$ are $n$ points such that
	each component of $(g(G^{'}) \cup \partial R(g)) \setminus N$ includes one of them exactly.
	$\{\tilde{p}_k\} = g^{-1}(p_k) \cap (\partial \Sigma_0 \cup G^{'}), \forall k \in \{1,2,\ldots,n\}$.
	Assume $G$ is an arbitrary $(R(g),g(G^{'}),N)$-thin trivalent graph.
	Then the construction triple $(\{p_1,p_2,\ldots,p_n\},G,$ $\{\tilde{p}_1,\tilde{p}_2,\ldots,\tilde{p}_n\})$
	is suitable,
	and the cancellable domains given by it are
	independent of the choice of $p_1,p_2,\ldots,p_n,\tilde{p}_1,\tilde{p}_2,\ldots,\tilde{p}_n$.
\end{lm}

\begin{proof}
	Fix $k \in \{1,2,\ldots,n\}$.
	Let $B_k$ be the component of $R(g) \setminus G$ containing $p_k$.
	Let $\tilde{B}_k$ be the component of $g^{-1}(B_k)$ containing $\tilde{p}_k$.
	
	First,
	we prove that $\tilde{B}_k$ is mapped homeomorphically to $B_k$ by $g$.
	Assume $l_k$ is the component of $(g(G^{'}) \cup \partial R(g)) \setminus N$ such that
	$l_k \subseteq B_k$.
	Let $\tilde{l}_k = g^{-1}(l_k) \cap (\partial \Sigma_0 \cup G^{'})$,
	then $\tilde{l}_k \subseteq \tilde{B}_k$.
	We cut off $g(G^{'}) \cap B_k$ from $B_k$ and obtain $B^{'}_{k}$.
	Similar to the proof of Lemma \ref{case 1 cancellation},
	for each component $B_0$ of $B^{'}_{k}$,
	if $\tilde{B}_0$ is the component of $g^{-1}(B_0)$ such that $\tilde{B}_0 \subseteq \tilde{B}_k$,
	then $\tilde{B}_0$ is mapped homeomorphically to $B_0$ by $g$.
	
	Easily,
	$\tilde{B}_1,\ldots,\tilde{B}_n$ are cancellable,
	and they are independent of the choice of $p_1,p_2,\ldots,p_n,\tilde{p}_1,\tilde{p}_2,$ $\ldots,\tilde{p}_n$.
\end{proof}

With above conditions,
the cancellable domains given by
$(\{p_1,p_2,\ldots,p_n\},G,\{\tilde{p}_1,\tilde{p}_2,\ldots,\tilde{p}_n\})$
(which are determined by $G$)
are called
the \emph{$(G,g(G^{'}))$-cancellable domains}.
Hence $G$ and $g(G^{'})$ define a cancellation $(g,\Sigma_0) \stackrel{(G,g(G^{'}))}{\leadsto} (g_1,\Sigma_1)$.
Moreover,
the cancellation is determined uniquely by $G$ and $g(G^{'})$ ($G^{'}$ is determined uniquely by $g(G^{'})$, since (v) of Case \ref{case 2}).

\begin{lm}\label{case 2 subgraph}\rm
Assume $A_1,A_2,\ldots,A_n$ are $(G,g(G^{'}))$-cancellable domains, then $G(A_1,A_2,\ldots,A_n) \in sub_{g(G^{'})}(G)$.
\end{lm}
	
	\section{Inscribed set}\ \label{section 5}
	
	This section introduces a way to realize the extensions of an immersed circle in the surfaces.
	Fix an immersed circle in the surface and a normal numbering $\psi$ of it,
	\emph{inscribed maps} (Definition \ref{inscribed map}) are maps established by some graphs under certain conditions.
	We define an \emph{inscribed set} $\zeta$ (Definition \ref{inscribed set})
	and $I(\zeta) \subseteq \zeta$ (Definition \ref{good graph set}).
	The elements of $I(\zeta)$ can be realized to extensions of the immersed circle by inscribed maps (Lemma \ref{realization}).
	
	\begin{defn}[Inscribed map]\label{inscribed map}\rm
		For a closed oriented surface $\Sigma$,
		let $f: S^{1} \to \Sigma$ be a homologically trivial immersion.
		Let $\psi$ be a normal numbering of $f$.
		Let $n = \max \psi$.
		Let $G_2 \subseteq D_2(f,\psi), G_3 \subseteq D_3(f,\psi),\ldots, G_n \subseteq D_n(f,\psi)$ be embedded graphs such that $V_k(f,\psi) = G_k \cap \partial D_k(f,\psi) = \{v \mid deg_{G_k}(v) = 1\}$,
		and $G_k \cap G_{k-1} = V(G_k) \cap V(G_{k-1})$, $\forall k \in \{2,3,\ldots,n\}$ ($G_1 = \emptyset$).
		Assume $g_k: D_k \to \Sigma$ is an embedding such that $g_k(D_k) = D_k(f,\psi)$, $\forall k \in \{1,2,\ldots,n\}$.
		There are graphs $A_k, B_k \subseteq D_k$ such that
		$g_k(A_k) = G_k, g_k(B_k) = G_{k+1}$,
		and $a_k: (V(A_k),E(A_k)) \to (V(G_k),E(G_k))$, $b_k: (V(B_k),E(B_k)) \to (V(G_{k+1}),E(G_{k+1}))$ are isomorphisms induced by $g_k$,
		$\forall k \in \{1,2,\ldots,n\}$ ($G_{n+1} = \emptyset$).
		Let $F = \coprod_{k=1}^{n} D_k$
		and $g: F \to \Sigma$ a map such that $g \mid_{D_k} = g_k, \forall k \in \{1,2,\ldots,n\}$.
		We obtain a map by following procedure:
		
		We cut off (which means to delete the set from the space and do a path compactification) $(\bigcup_{k=2}^{n}A_k) \cup (\bigcup_{k=1}^{n-1}B_k)$ from $F$,
		and obtain an identification space $F_0$.
		$g_0: F_0 \to \Sigma$ is induced by $g: F \to \Sigma$.
		For all $k \in \{2,\ldots,n\}$ and $e \in E(G_k)$,
		$a^{-1}_{k}(e)$ is cut off in $D_{k}$
		and becomes $2$ copies $e_{1}^{+}$ (in the left) and $e_{1}^{-}$ (in the right),
		$b^{-1}_{k-1}(e)$ is cut off in $D_{k-1}$
		and becomes $e_{2}^{+}$ (in the left) and $e_{2}^{-}$ (in the right).
		Let $h$ be the equivalence relation such that $x \stackrel{h}{\sim} y$ if
		$x \in e_{1}^{+}, y \in e_{2}^{-}, g_0(x) = g_0(y)$ or $x \in e_{2}^{+}, y \in e_{1}^{-}, g_0(x) = g_0(y)$,
		$\forall e \in \bigcup_{k=2}^{n} E(G_k)$.
		Let $F_1 = F_0 /\sim_h$,
		and $g_1: F_1 \to \Sigma$ be induced by $g$.
		$g_1$ is called an \emph{inscribed map} of $(f,\psi)$
		\emph{associated} to $\{G_2,\ldots,G_n\}$.
	\end{defn}

Actually,
the inscribed map $g_1: F_1 \to \Sigma$ is a polymersion (but we can't ensure it's different critical points mapped to different branch points).
$F_1$ is a surface that may be disconnected,
and $g_1 \mid_{\partial F_1} = f$.
We construct the sets of graphs whose inscribed maps are immersions of connected surfaces in the remainder of this subsection.

	\begin{defn}[Inscribed set]\label{inscribed set}\rm 
		For a closed oriented surface $\Sigma$,
		let $f: S^{1} \to \Sigma$ be a homologically trivial immersion 
		and $\psi$ a normal numbering of $f$.
		Let $n = \max \psi$.
		The following process is to obtain 
		an \emph{inscribed set} $\zeta_1$ and the \emph{$k$th-inscribed set} $\zeta_k$ ($2 \leqslant k \leqslant n$):
		
		We induce decreasingly on $k$.
		For $k = n$:
		if $(D_n(f,\psi),V_n(f,\psi))$ is appropriate,
		then there exists $\tilde{H}_n$ a $(D_n(f,\psi),V_n(f,\psi))$-trivalent graph.
		Set $\zeta_n = \{(\tilde{H}_n,H_n) \mid H_n \in sub(\tilde{H}_n)\}$.
		If $(D_n(f,\psi),V_n(f,\psi))$ is not appropriate,
		then we set $\zeta_n = \emptyset$.
		
		Now assume we have defined $\zeta_n, \zeta_{n-1}, \ldots, \zeta_{k+1}$ after $n-k$ steps,
		$1 \leqslant k \leqslant n-1$.
		For step $n-k+1$:
		if $\zeta_{k+1} = \emptyset$,
		set $\zeta_{k} = \emptyset$.
		If $\zeta_{k+1} \ne \emptyset$,
		then $\zeta_k$ is obtained as follows:
		for each $=\{(\tilde{H}_{k+1},H_{k+1}), \ldots, (\tilde{H}_n,H_n)\} \in \zeta_{k+1}$,
		let $N_j = \{x \mid x \in V(H_j), deg(x) = 3\} \setminus N_{j+1}$,
		$\forall k+1 \leqslant j \leqslant n$ ($N_{n+1} = \emptyset$).
		Then:
		
		$\bullet$
		If $(D_k(f,\psi),H_{k+1},N_{k+1} \cup V_k(f,\psi))$ is appropriate,
		set $\tilde{H}_k$ to be a thin $(D_k(f,\psi),H_{k+1},N_{k+1} \cup V_k(f,\psi))$-trivalent graph.
		Let
		$Q(\{(\tilde{H}_{k+1},H_{k+1}), \ldots, (\tilde{H}_n,H_n)\}) =
		\{(\tilde{H}_k,H_k) \mid H_k \in sub_{H_{k+1}}(\tilde{H}_k)\}$.
		
		$\bullet$
		If $(D_k(f,\psi),H_{k+1},N_{k+1} \cup V_k(f,\psi))$ is not appropriate,
		set $Q(\{(\tilde{H}_{k+1},H_{k+1}), \ldots, (\tilde{H}_n,H_n)\}) = \emptyset$.
		
		Set $\zeta_k = \bigcup_{X \in \zeta_{k+1}, Q(X) \ne \emptyset} \bigcup_{Y \in Q(X)} (X \cup Y)$.
	\end{defn}
	
	\begin{remark}\label{remark of graph presentation}\rm
	    For simplicity,
	    we will always denote an inscribed set by $\zeta$ instead of $\zeta_1$,
	    and denote by $\zeta_k$ the $k$th-inscribed set obtained by the procedure to obtain $\zeta$ (determined by $\zeta$ uniquely),
	    $\forall k \in \{2,3,\ldots,n\}$.
	\end{remark}

	\begin{defn}\label{good graph set}\rm
		For a closed oriented surface $\Sigma$,
		let $f: S^{1} \to \Sigma$ be a homologically trivial immersion 
		and $\psi$ a normal numbering of $f$.
		Let $\zeta$ be an inscribed set of $(f,\psi)$.
		An element $\{(\tilde{H}_1,H_1),\ldots,(\tilde{H}_n,H_n)\} \in \zeta$ is \emph{good}
		if $H_1 = \emptyset$,
		and $H_2,H_3,\ldots,H_n \ne \emptyset$ (when $n \geqslant 1$).
		Let $I(\zeta) = \{X \mid X \in \zeta,X$ $is$ $good\}$.
	\end{defn}

For each $\{(\tilde{H}_1,H_1),\ldots,(\tilde{H}_n,H_n)\} \in I(\zeta)$,
we can verify that
$g: S \to \Sigma$ is an inscribed map of $f$ associated to $H_2,\ldots,H_n$
if and only if
there exists a sequence of cancellation operations
$(g,S) \stackrel{\tilde{H}_n}{\leadsto} (g_n,S_n) 
\stackrel{(\tilde{H}_{n-1},H_n)}{\leadsto} 
(g_{n-1},S_{n-1}) \stackrel{(\tilde{H}_{n-2},H_{n-1})}{\leadsto} \ldots 
\stackrel{(\tilde{H}_{1},H_2)}{\leadsto} (g_1,S_1)$
(where $(g_k,S_k) \stackrel{(\tilde{H}_{k-1},H_k)}{\leadsto} (g_{k-1},S_{k-1})$ is the cancellation of $(\tilde{H}_{k-1},H_k)$-cancellable domains in $g_k$, 
and $g_1$ is an embedding).

\begin{lm}\label{realization}
	If $\{(\tilde{H}_1,H_1),\ldots,(\tilde{H}_n,H_n)\} \in I(\zeta)$,
	$g: S \to \Sigma$ is an inscirbed map of $(f,\psi)$ associated to $\{H_2,\ldots,H_n\}$,
	then $g$ is an immersion.
	Moreover,
	$g$ is an extension of $f$ related to $\psi$ (which means $S$ is a connected surface).
\end{lm}

\begin{proof}
	We only prove that $S$ is a connected surface.
	Let $h$ be the equivalence class consistent with Definition \ref{inscribed map}
	and $h_*: F_0 \to F_1$ the identification map ($h_*$ is surjective).
	We denote by $D^{'}_{k} \subseteq F_0$ the space obtained by cutting off $A_k \cup B_k$ from $D_k$,
	$\forall k \in \{1,2,\ldots,n\}$.
	
	Assume $X$ is a component of $D_k(f,\psi) \setminus (H_k \cup H_{k+1})$,
	$k \in \{1,2,\ldots,n\}$ ($H_{n+1} = \emptyset$),
	then $\overline{X} \cap (H_{k+1} \cup \partial D_k(f,\psi)) \ne \emptyset$.
	So for each $x \in D^{'}_{k}$,
	$h_*(x)$ is connected to a point in $h_*(D^{'}_{k+1}) \cup \partial S$,
	$k \in \{1,2,\ldots,n\}$ ($D^{'}_{n+1} = \emptyset$).
	Hence each point in $h_*(F_0) = F_1$ is connected to $\partial S$.
	So $S$ is connected, since $\partial S$ has exactly one boundary component.
\end{proof}

$g$ is said to be an inscribed map of $(f,\psi)$ \emph{to realize} $\{(\tilde{H}_1,H_1),\ldots,(\tilde{H}_n,H_n)\}$,
or an inscribed map of $(f,\psi)$ \emph {related to} $\{(\tilde{H}_1,H_1),\ldots,(\tilde{H}_n,H_n)\}$.
Moreover,
Lemma \ref{realization} defines a map from $I(\zeta)$ to the set of equivalence classes of extensions of $f$ related to $\psi$.
   
    \section{The proof of Theorem 1}\ \label{section 6}
    
    We prove Theorem \ref{main} in this section.
    Given a closed oriented surface $\Sigma$ and an immersion $f: S^{1} \to \Sigma$.
    Let $\psi$ be a normal numbering of $f$
    and $\zeta$ an inscribed set  of $(f,\zeta)$.
    Let $E(f,\psi)$ be the set of equivalence classes of extensions of $f$ related to $\psi$.
    Lemma \ref{realization} provides a map $i: I(\zeta) \to E(f,\psi)$ sending each $X \in I(\zeta)$ to the equivalence class of the inscribed map of $(f,\psi)$ related to $X$.
    We prove $i: I(\zeta) \to E(f,\psi)$ is a bijection in this section.
    Lemma \ref{inequivalent} proves $i: I(\zeta) \to E(f,\psi)$ is injective,
    and Proposition \ref{surjective} proves $i: I(\zeta) \to E(f,\psi)$ is surjective.
    Lemma \ref{realization} and Proposition \ref{surjective}
    conclude $i: I(\zeta) \to E(f,\psi)$ is a bijection,
    hence Theorem \ref{main} is proved.
    
    \begin{lm}\label{inequivalent}
    	For a closed oriented surface $\Sigma$,
    	let $f: S^{1} \to \Sigma$ be a homologically trivial immersion
    	and $\psi$ a normal numbering of $f$. 
    	Assume $n = \max \psi$.
    	If $\zeta$ is an inscribed set of $(f,\psi)$,
    	and $g_1,g_2: S \to \Sigma$ are two inscribed map of $(f,\psi)$ related to two different elements of $I(\zeta)$,
    	then $g_1,g_2$ are inequivalent.
    \end{lm}

\begin{proof}
	Suppose $g_1$ is related to $\{(\tilde{H}_1,H_1),\ldots,(\tilde{H}_n,H_n)\} \in I(\zeta)$,
	$g_2$ is related to $\{(\tilde{G}_1,G_1),\ldots,(\tilde{G}_n,G_n)\} \in I(\zeta)$.
	Then $\{(\tilde{H}_1,H_1),\ldots,(\tilde{H}_n,H_n)\} \ne \{(\tilde{G}_1,G_1),\ldots,(\tilde{G}_n,G_n)\}$.
	There exists $k \in \{2,3,\ldots,n\}$
	such that $H_k \ne G_k$ and $H_i = G_i$, $\forall k+1 \leqslant i \leqslant n$.
	So $\tilde{H}_k = \tilde{G}_k$.
	There exists $e_0 \in E(\tilde{H}_k)$,
	$e_0$ is in exactly one of $E(H_k),E(G_k)$.
	Assume without loss of generality that $e_0 \in E(H_k), e_0 \notin E(G_k)$.
	
	For each $e \in E(H_i) \subseteq E(\tilde{H_i})$, $\forall i \in \{2,\ldots,n\}$,
	we denote by $D_+(e)$ (respectively $D_-(e)$) the closure of the component of $D_i(f,\psi) \setminus (\tilde{H_i} \cup H_{i+1})$ which lie in the left side (respectively the right side) of $e$.
	Then $\partial D_+(e) \cap (H_{i+1} \cup \partial D_{i}(f,\psi)), \partial D_-(e) \cap (H_{i+1} \cup \partial D_{i}(f,\psi)) \ne \emptyset$.
	
	Choose $s \in e_0$.
	There exist $m \in \{k,k+1,\ldots,n\}$ and $p_1 \in \partial D_m(f,\psi) \setminus V_m(f,\psi)$,
	such that $\exists a_0 = e_0, a_1 \in E(H_{k+1}), a_2 \in E(H_{k+2}), \ldots, a_{m-k} \in E(H_m), a_{m-k+1} = \{p_1\}$, 
	$a_{i+1} \subseteq D_{+}(a_{i}), \forall i \in \{0,1,2,\ldots,m-k\}$.
	Let $h_1: [0,1] \to \Sigma$ be an immersion such that
	$h_1(0) = s, h_1(1) = p_1$,
	and $\exists 0=t_0<t_1<t_2<\ldots<t_{m-k}<t_{m-k+1}=1$ such that
	$h_1(t_i) \in a_i$, 
	and $[t_i,t_{i+1}]$ is homeomorphically embedded into $D_+(a_i)$ by $h_1$,
	$\forall i \in \{0,1,\ldots,m-k\}$.
	Similarly, there exist $q \in \{k,k+1,\ldots,n\}$ and $p_2 \in \partial D_q(f,\psi) \setminus V_q(f,\psi)$,
	such that $\exists b_0 = e_0, b_1 \in E(H_{k+1}), b_2 \in E(H_{k+2}), \ldots, b_{q-k} \in E(H_q), b_{q-k+1} = \{p_2\}$, 
	$b_{i+1} \subseteq D_{-}(b_{i}), \forall i \in \{0,1,2,\ldots,q-k\}$.
	And let $h_2: [0,1] \to \Sigma$ be an immersion such that
	$h_2(0) = s, h_2(1) = p_2$,
	and $\exists 0=j_0<j_1<j_2<\ldots<j_{q-k}<j_{q-k+1}=1$ such that
	$h_2(j_i) \in b_i$, 
	and $[j_i,j_{i+1}]$ is homeomorphically embedded into $D_-(b_i)$ by $h_2$,
	$\forall i \in \{0,1,\ldots,q-k\}$.
	
	Since $g_1,g_2$ are the inscribed maps related to $\{(\tilde{H}_1,H_1),\ldots,(\tilde{H}_n,H_n)\}$ and $\{(\tilde{G}_1,G_1),\ldots,(\tilde{G}_n,G_n)\}$,
	there exsit embeddings $\tilde{h}_{1},\tilde{h}_{2}: [0,1] \to S$ such that $g_1 \circ \tilde{h}_1 = g_2 \circ \tilde{h}_2 = h_1$,
	$\{\tilde{h}_{1}(1)\} = \{\tilde{h}_{2}(1)\} = \partial S \cap g^{-1}_{1}(p_1)$,
	and there exist embeddings $\tilde{h}_{3},\tilde{h}_{4}: [0,1] \to S$ such that $g_1 \circ \tilde{h}_3 = g_2 \circ \tilde{h}_4 = h_2$,
	$\{\tilde{h}_{3}(1)\} = \{\tilde{h}_{4}(1)\} = \partial S \cap g^{-1}_{1}(p_2)$.
	Since $e_0 \in E(H_k)$ and $e_0 \notin E(G_k)$,
	$\tilde{h}_{1}([0,1]) \cap \tilde{h}_{3}([0,1]) = \emptyset$,
	but $\tilde{h}_{2}([0,1]) \cap \tilde{h}_{4}([0,1]) \ne \emptyset$.
	So there is a properly embedded arc $\tilde{h}_{2}([0,1]) \cap \tilde{h}_{4}([0,1])$ immersed to $h_1([0,1]) \cap h_2([0,1])$ under $g_2$,
	and there is no properly embedded arc immersed to $h_1([0,1]) \cap h_2([0,1])$ under $g_1$.
	Hence $g_1,g_2$ are inequivalent.
\end{proof}
    
\begin{prop}\label{surjective}
	Let $\Sigma$ be a closed oriented surface,
	and let $\Sigma_0$ be a compact orientable surface such that $\partial \Sigma_0$ has exactly one component.
	Let $g: \Sigma_0 \to \Sigma$ be an immersion.
	Let $\psi$ be the normal numbering of $g \mid_{\partial \Sigma_0}$ given by $g$.
	Let $n = \max \psi$.
	Let $\zeta$ be an inscribed set of $(g \mid_{\partial \Sigma_0},\psi)$.
	Then
        $I(\zeta) \ne \emptyset$,
    	and there exists $\{(\tilde{H}_1,H_1),\ldots,(\tilde{H}_n,H_n)\} \in I(\zeta)$ 
    	such that $g$ is the inscribed set of $g \mid_{\partial \Sigma_0}$ related to $\{(\tilde{H}_1,H_1),\ldots,(\tilde{H}_n,H_n)\}$.
    \end{prop}
	
	\begin{proof}
		Let $f: \partial \Sigma_0 \to \Sigma$ be $g \mid_{\partial \Sigma_0}$,
		then $f$ is an homologically trivial immersion.
		
		To prove that there exists $\{(\tilde{H}_1,H_1),\ldots,(\tilde{H}_n,H_n)\} \in I(\zeta)$
		such that
		$g$ is the inscribed map of $f$ related to it,
		our aim is to 
		construct a sequence of cancellation operations
		$(g,\Sigma_0) \stackrel{\tilde{H}_n}{\leadsto} (g_n,\Sigma_n) \stackrel{(\tilde{H}_{n-1},H_n)}{\leadsto} 
		(g_{n-1},\Sigma_{n-1}) \stackrel{(\tilde{H}_{n-2},H_{n-1})}{\leadsto} \ldots
		\stackrel{(\tilde{H}_{1},H_2)}{\leadsto} (g_1,\Sigma_1)$,
		where $(g_k,\Sigma_k) \stackrel{(\tilde{H}_{k-1},H_k)}{\leadsto} (g_{k-1},\Sigma_{k-1})$ is the cancellation of $(\tilde{H}_{k-1},H_k)$-cancellable domains in $g_k$.
		
		$R(g) = D_n(f,\psi)$.
		We know $(D_n(f,\psi)),V_n(f,\psi))$ is appropriate by Lemma \ref{case 1 appropriate}.
		$\tilde{H}_n$ is determined by $\zeta$,
		and we denote by $U_{(n,1)}, U_{(n,2)}, \ldots, U_{(n,t_n)}$ the $\tilde{H}_n$-cancellable domains (Definition \ref{case 1 cancellation}).
		Let $H_n$ be $G(U_{(n,1)}, U_{(n,2)}, \ldots,$ $U_{(n,t_n)}) \in sub(H_n)$ (Lemma \ref{case 1 subgraph}),
		then $\{(\tilde{H}_n,H_n)\} \in \zeta_n$.
		We cancel $\{U_{(n,1)}, U_{(n,2)}, \ldots, U_{(n,t_n)}\}$.
		Then obtain a space $\Sigma_n$ and a polymersion $g_n: \Sigma_n \to \Sigma$.
		Let $f_n: \partial \Sigma_n \to \Sigma$ be $g_n \mid_{\partial \Sigma_n}$.
		
		Let $N_n = \{x \mid x \in V(H_n), deg(x) = 3\}$,
		then $N_n$ is the set of branch points of $g_n$
		and each element of $N_n$ has index $1$.
		The cancellation of $\{U_{(n,1)}, U_{(n,2)}, \ldots, U_{(n,t_n)}\}$ is regular (Definition \ref{regular} (ii)),
		so there exists a map $h_n: H_n \to \Sigma^{'}_n$ 
		associated to the cancellation of
		$\{U_{(n,1)}, U_{(n,2)}, \ldots, U_{(n,t_n)}\}$.
		
		We now do the steps by induction.
		
		Assume that
		there exists $\{(\tilde{H}_{i+1},H_{i+1}),\ldots,(\tilde{H}_n,H_n)\} \in \zeta_{i+1}$,
		and we have gotten the cancellations
		$(g,\Sigma_0) \stackrel{\tilde{H}_n}{\leadsto} (g_n,\Sigma_n) \stackrel{(\tilde{H}_{n-1},H_n)}{\leadsto} 
		\ldots \stackrel{(\tilde{H}_{i+1},H_{i+2})}{\leadsto} (g_{i+1},\Sigma_{i+1})$,
		$1 \leqslant i \leqslant n-1$.
		Now $g_{i+1}: \Sigma_{i+1} \to \Sigma$ is a polymersion of the compact orientable surface $\Sigma_{i+1}$ such that
		$R(g_{i+1}) = D_i(f,\psi)$.
		Let $f_{i+1}: \partial \Sigma_{i+1} \to \Sigma$ be $g_{i+1} \mid_{\partial \Sigma_{i+1}}$.
		
		We induce as follows:
		
		(i)
		Let $N_j = \{x \mid x \in V(H_j), deg(x) = 3\} \setminus N_{j+1}$,
		$\forall j \in \{i+1,\ldots,n\}$ ($N_{n+1} = \emptyset$).
		$N_{i+1}$ is the set of branch points of $g_{i+1}$,
		and each element of it has index $1$.
		
		(ii)
		There is an embedding $h_{i+1}: H_{i+1} \to \Sigma_{i+1}$ such that:
		\begin{center}
			$g_{i+1} \circ h_{i+1} = id$,
		\end{center}
		\begin{center}
			$h_{i+1} (y) \in \partial \Sigma_{i+1}$, $\forall y \in H_{i+1} \cap \partial D_{i+1}(f,\psi)$,
		\end{center}
	    and $h_{i+1}(N_{i+1})$ is the set of critical points of $g_{i+1}$.
		
		(iii)
		For each $x \in N_{i+1}$,
		let $t = h_{i+1}(x)$,
		which is the critical point of multiplicity $1$ mapped to $x$.
		Let $a,b,c$ be the $3$ edges of $H_{i+1}$ at $x$ clockwise.
		Then $h_{i+1}(a),h_{i+1}(c),h_{i+1}(b)$ are clockwise around $t$.
		
		\begin{figure}\label{$g_{i+1}$ near $x$}
			\centering 
			\includegraphics[width=0.5\textwidth]{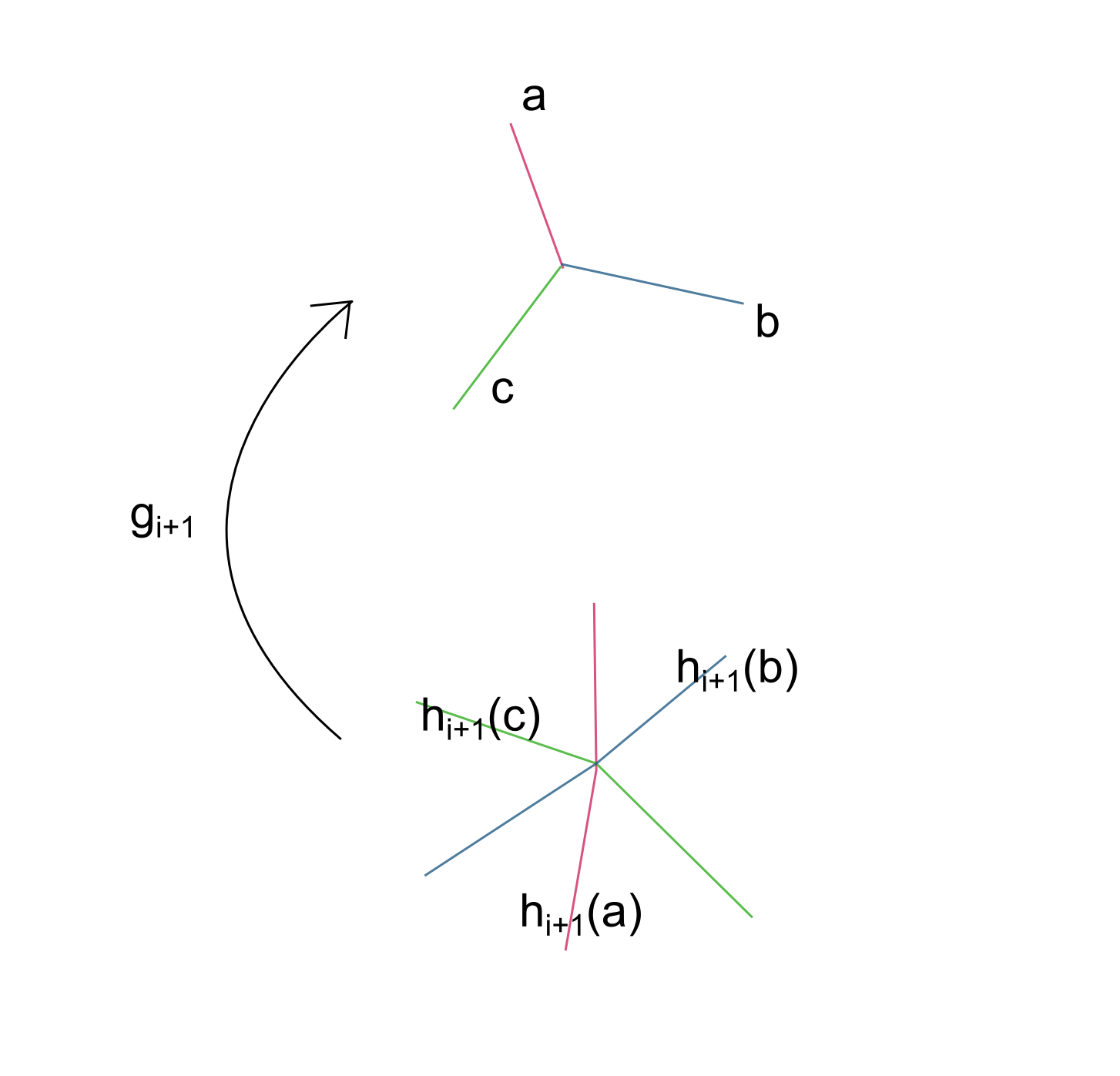}
			\caption{$g_{i+1}$ near $x$.}
		\end{figure}
		
		We know $(D_i(f,\psi),H_{i+1},N_{i+1} \cap V_i(f,\psi))$ is appropriate by Lemma \ref{case 2 appropriate}.
		The graph $\tilde{H}_i$ is determined by $\{(\tilde{H}_{i+1},H_{i+1}),\ldots,(\tilde{H}_n,H_n)\} \in \zeta_{i+1}$.
		We denote by $U_{(i,1)}, U_{(i,2)}, \ldots, U_{(i,{t_i})} \subseteq \Sigma_{i+1}$ the $\tilde{H}_i$-cancellable domains (Definition \ref{case 2 cancellation}). 
		Let $H_i$ be $G(U_{(i,1)}, U_{(i,2)}, \ldots, U_{(i,{t_i})}) \in sub_{H_{i+1}}(\tilde{H}_i)$ (Lemma \ref{case 2 subgraph}), 
		then $\{(\tilde{H}_i,H_i), (\tilde{H}_{i+1},H_{i+1}) \ldots, (\tilde{H}_n,H_n)\} \in \zeta_i$. 
		We cancel $\{U_{(i,1)}, U_{(i,2)}, \ldots, U_{(i,{t_i})}\}$,
		and obtain a surface $\Sigma_i$ and a polymersion
		$g_i: \Sigma_i \to \Sigma$.
		Let $f_i: \partial \Sigma_i \to \Sigma = g_i \mid_{\partial \Sigma_i}$.
		
	    Note that the elements of $N_{i+1}$ are not branch points of $g_i$ (induction hypothesis (iii)),
	    and the cancellation of $\{U_{(i,1)}, U_{(i,2)}, \ldots, U_{(i,{t_i})}\}$ is regular.
		So there exists an associated map $h_i: H_i \to \Sigma_i$ 
		of canceling
		$\{U_{(i,1)}, U_{(i,2)}, \ldots, U_{(i,t_i)}\}$.
		Let $N_i = \{x \mid x \in V(H_i), deg(x) = 3\} \setminus N_{i+1}$,
		then $N_i$ is the set of branch points of $g_i$,
		and each element of $N_i$ has index $1$.
		Also,
		the images of edges of $H_i$ at $N_i$ have positions in accordance with induction hypothesis (iii).
		So the induction hypothesises (i), (ii), (iii) are established when $k = n$.\
		
		We construct cancellations
		$(g,\Sigma_0) \stackrel{\tilde{H}_n}{\leadsto} (g_n,\Sigma_n) \stackrel{(\tilde{H}_{n-1},H_n)}{\leadsto} \ldots
		\stackrel{(\tilde{H}_{1},H_2)}{\leadsto} (g_1,\Sigma_1)$
		by the induction.
		Next,
		we prove $\{(\tilde{H}_1,H_1),\ldots,(\tilde{H}_n,H_n)\} \in I(\zeta)$.
		
		Assume $r: A \to B$ is a polymersion ($A, B$ are surfaces).
		For each $x \in B$,
		let $d_r(x) = \#(\{r^{-1}(x)\}) + i(x))$ ($i(x) =$ the index of $x$ if $x$ is a branch point, otherwise $i(x) = 0$).
		Then $d_{g_{i+1}}(x) - d_{g_i}(x) = 1$ if $x \in D_i(f,\psi)$,
		$d_{g_{i+1}}(x) = d_{g_i}(x)$ if $x \notin D_i(f,\psi)$,
		$\forall i \in \{3,4,\ldots,n\}$.
		After $n-1$ steps,
		$d_{g_2}(x) = 1$, $\forall x \in D_1(f,\psi)$.
		So $\Sigma_2$ is homeomorphically embedded into $D_1(f,\psi)$ by $g_2$.
		This means $H_1 = \emptyset$.
		$\Sigma_0$ is connceted so $H_k = \emptyset$,
		$\forall k \geqslant 2$.
		So we get a good graph set $\{(\tilde{H}_1,H_1),\ldots,(\tilde{H}_n,H_n)\} \in \zeta$ from the steps above, 
		and $g$ is the inscribed map of $f$ related to it. 
	\end{proof}

\section{Summary}\ \label{section 7}

\begin{figure}\label{example 1}
	\centering
	\includegraphics[width=0.35\textwidth]{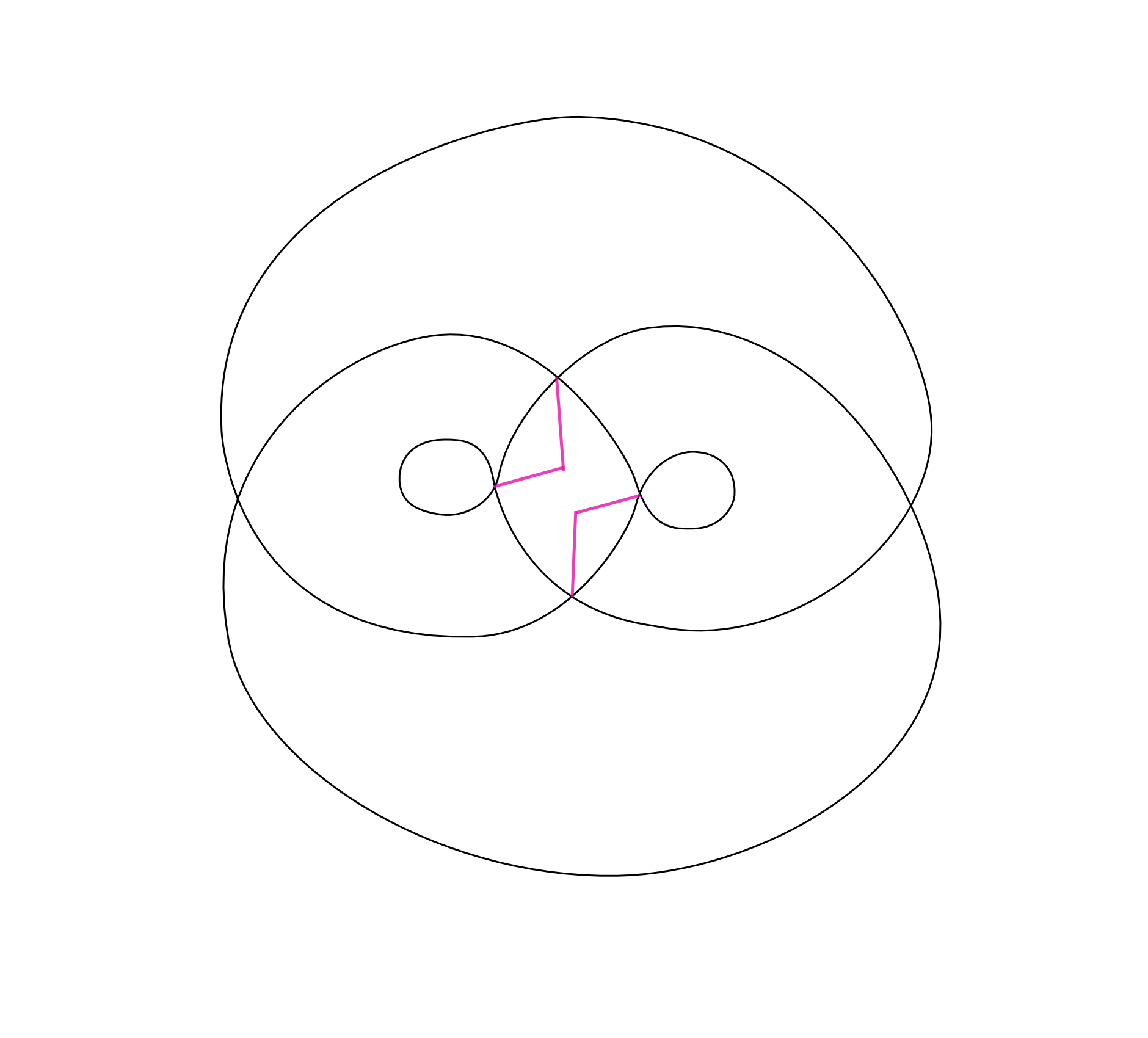}
	\includegraphics[width=0.35\textwidth]{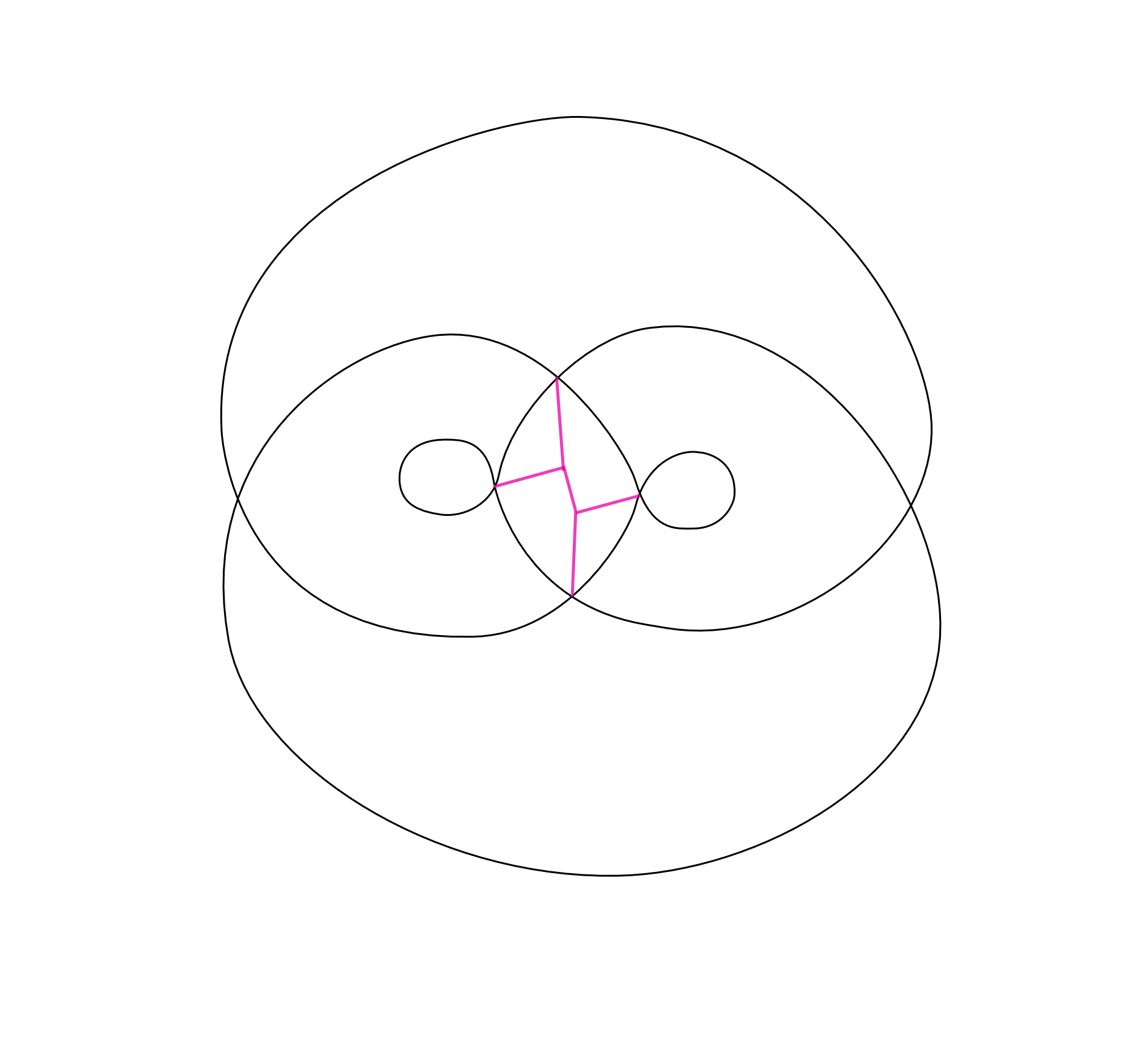}
	(a)
	\includegraphics[width=0.25\textwidth]{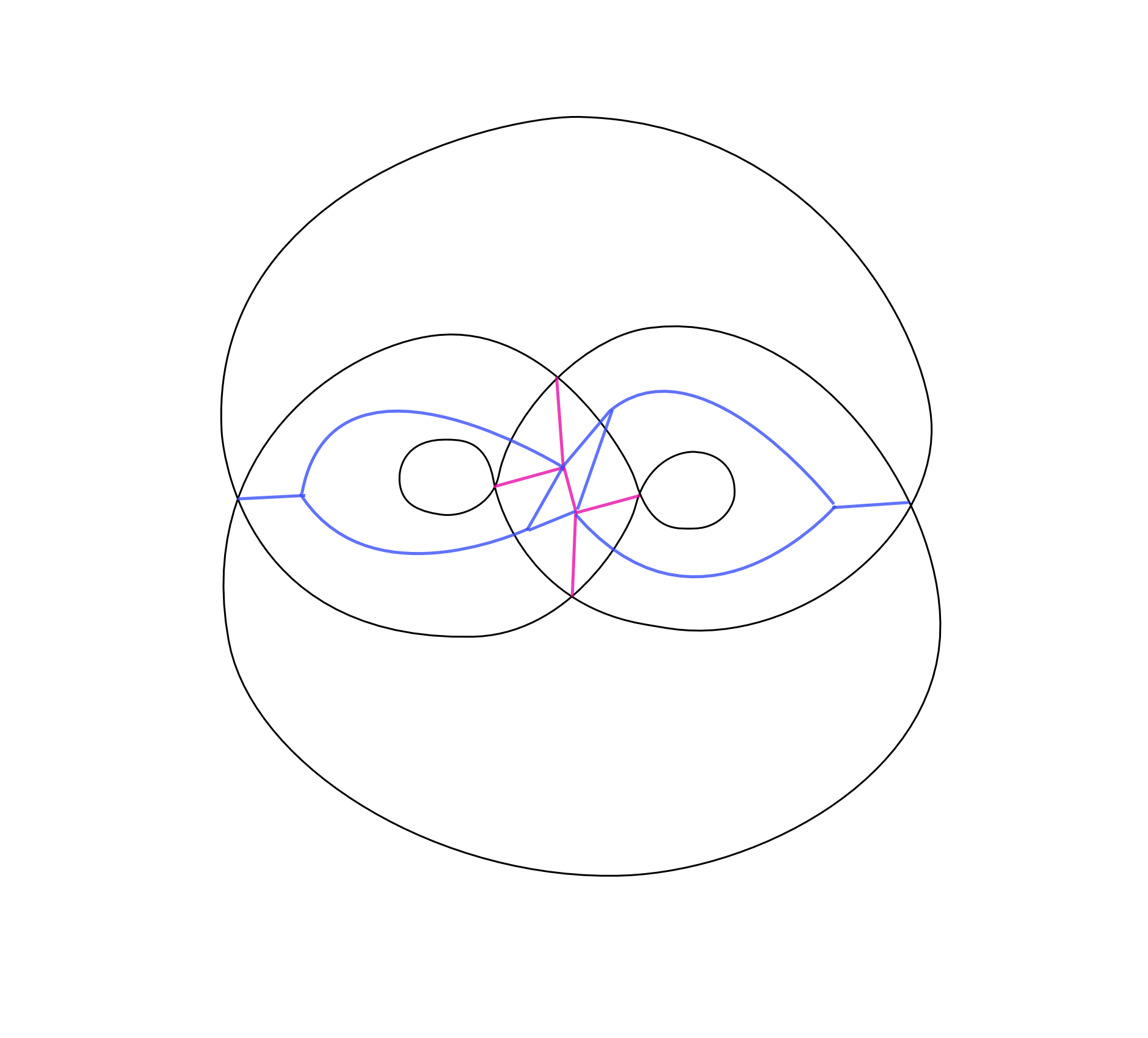}
	\includegraphics[width=0.25\textwidth]{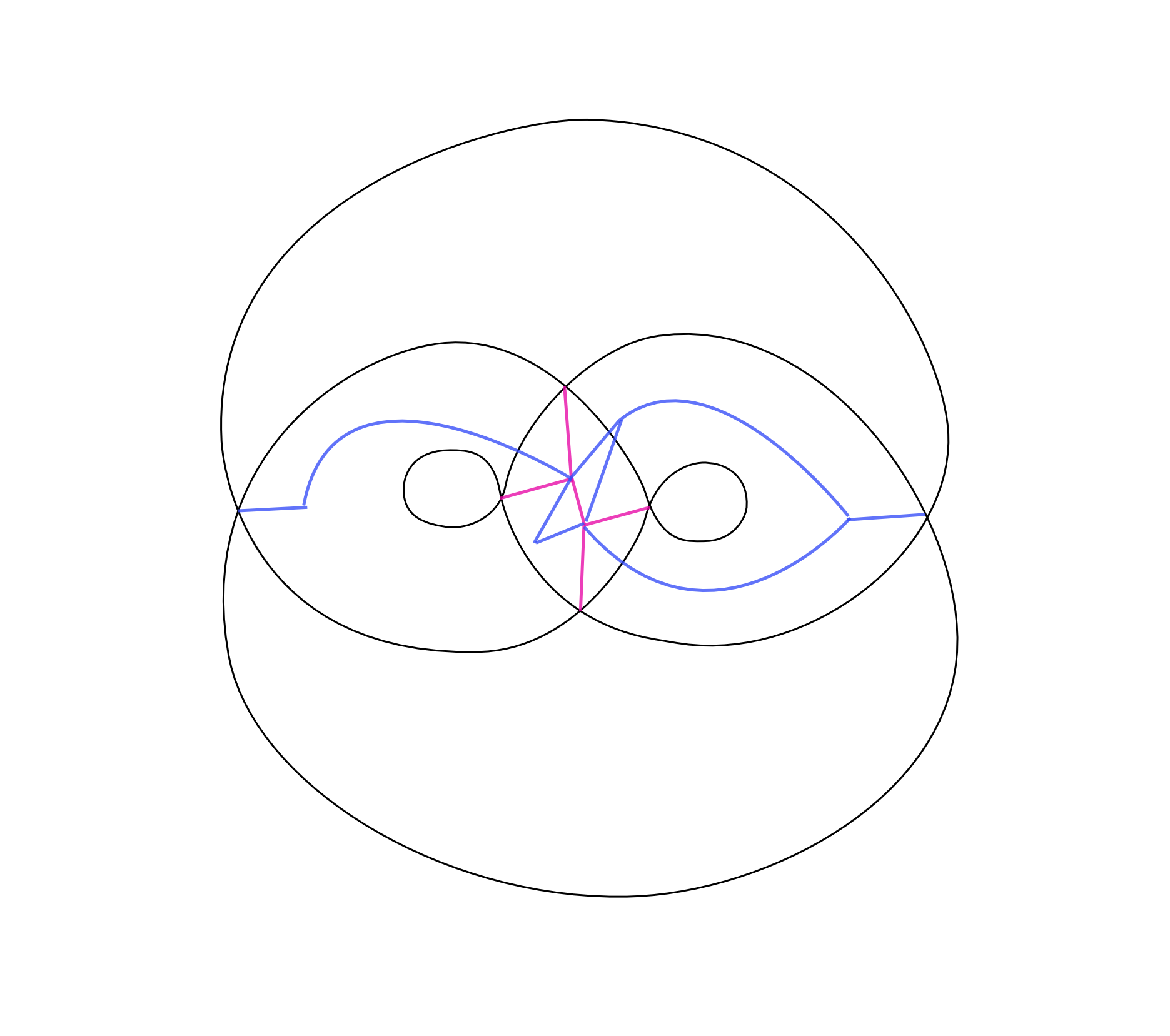}
	\includegraphics[width=0.25\textwidth]{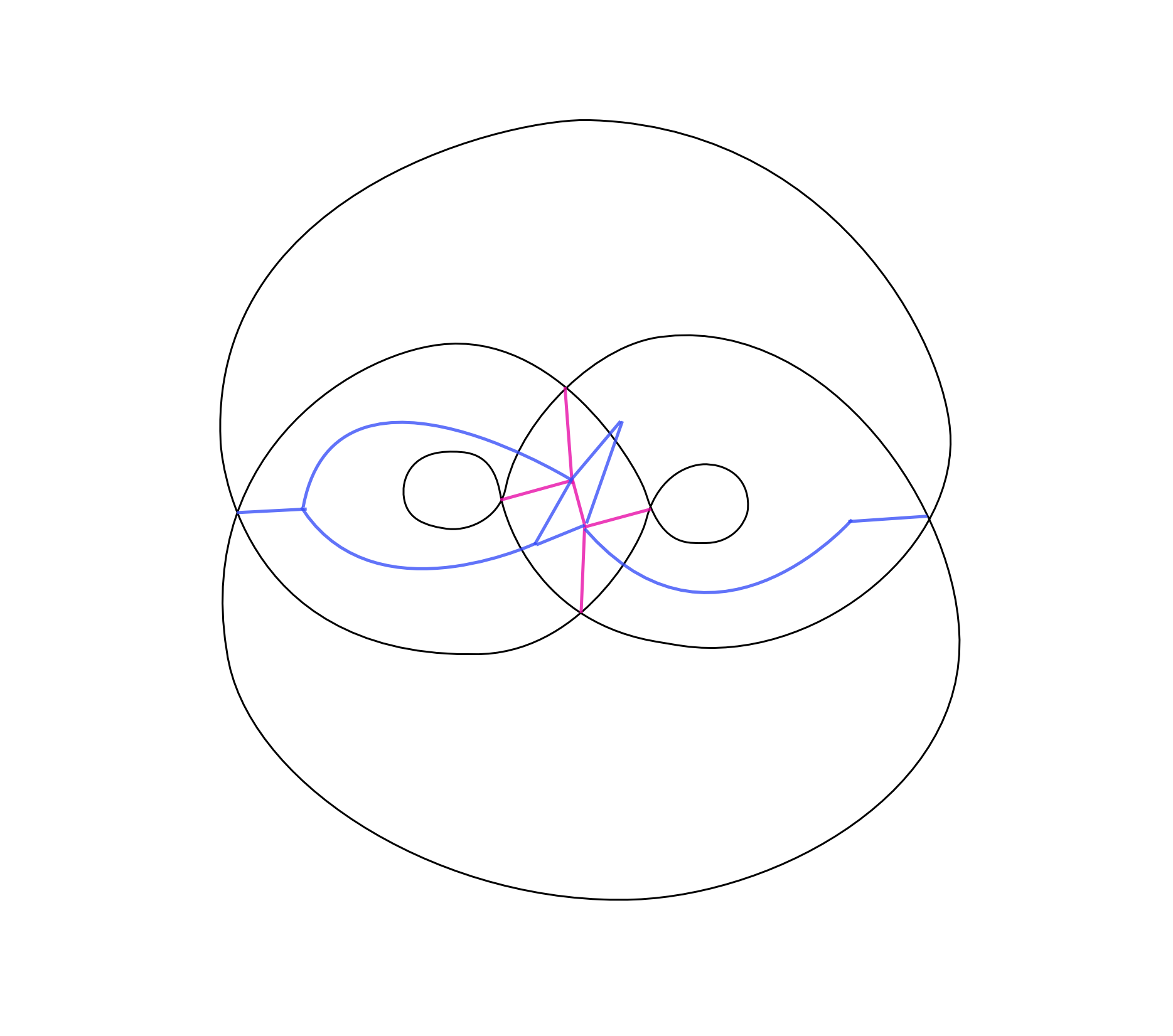}
	\includegraphics[width=0.25\textwidth]{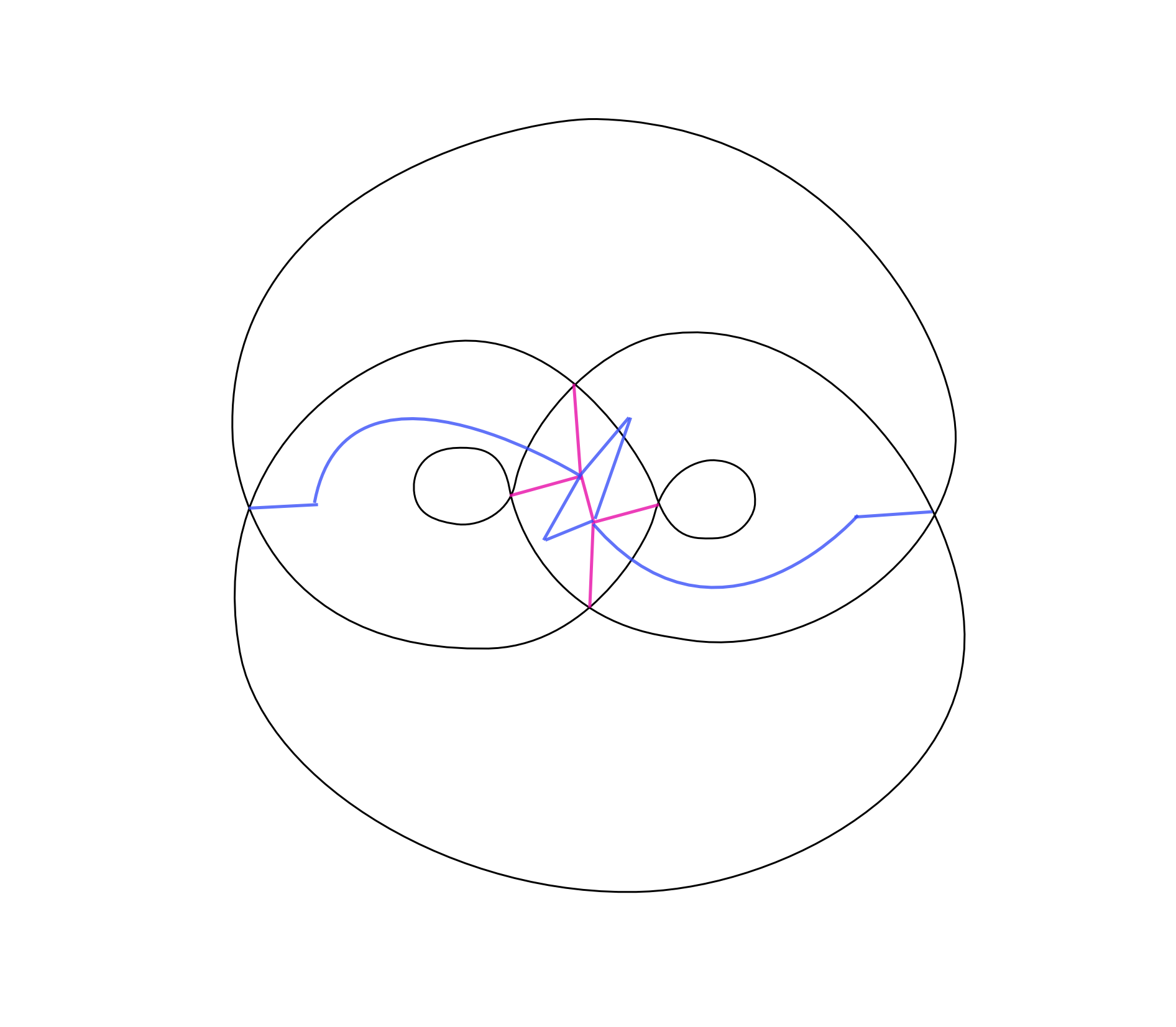}
	\includegraphics[width=0.25\textwidth]{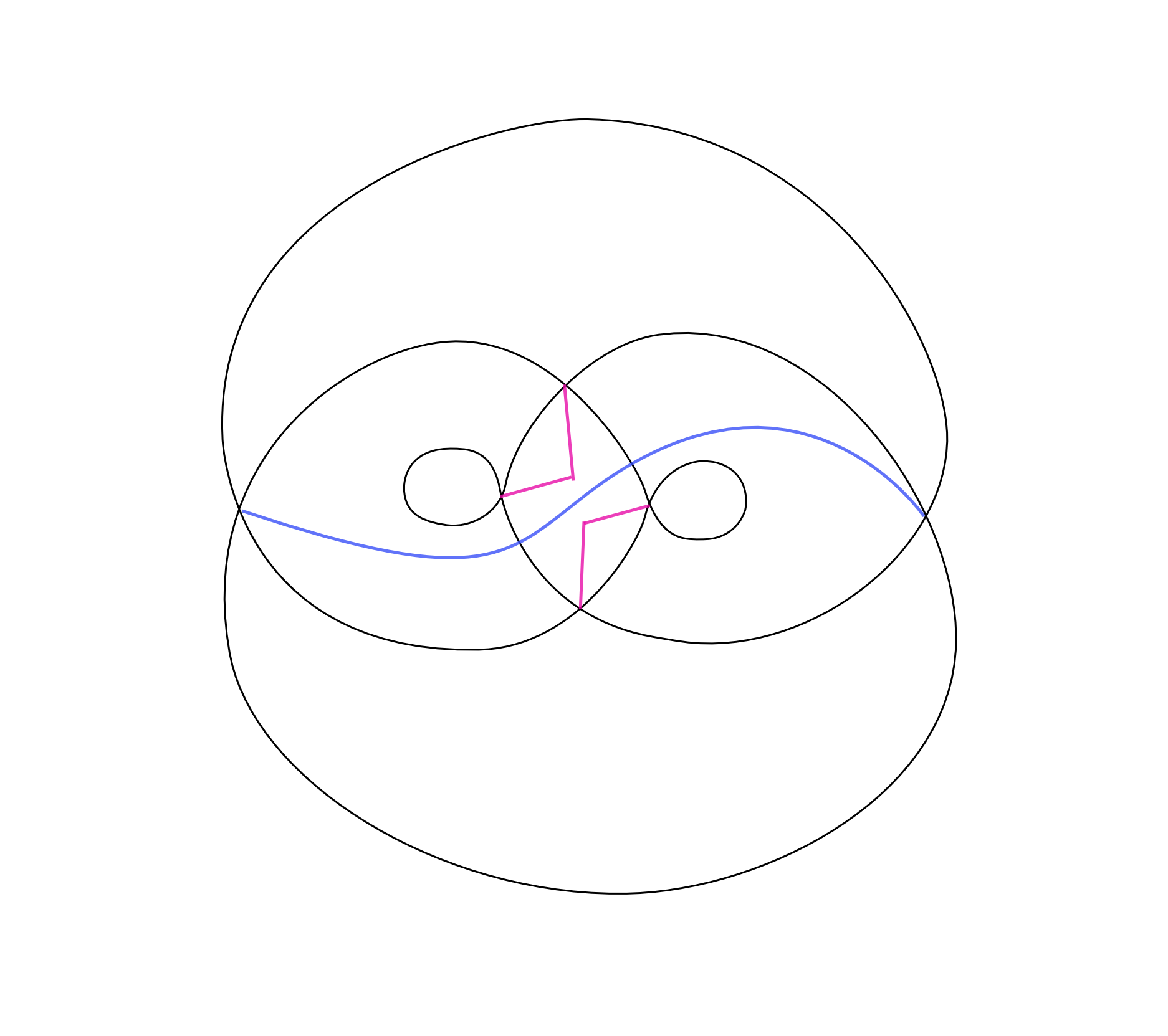}
	(b)
	\includegraphics[width=0.45\textwidth]{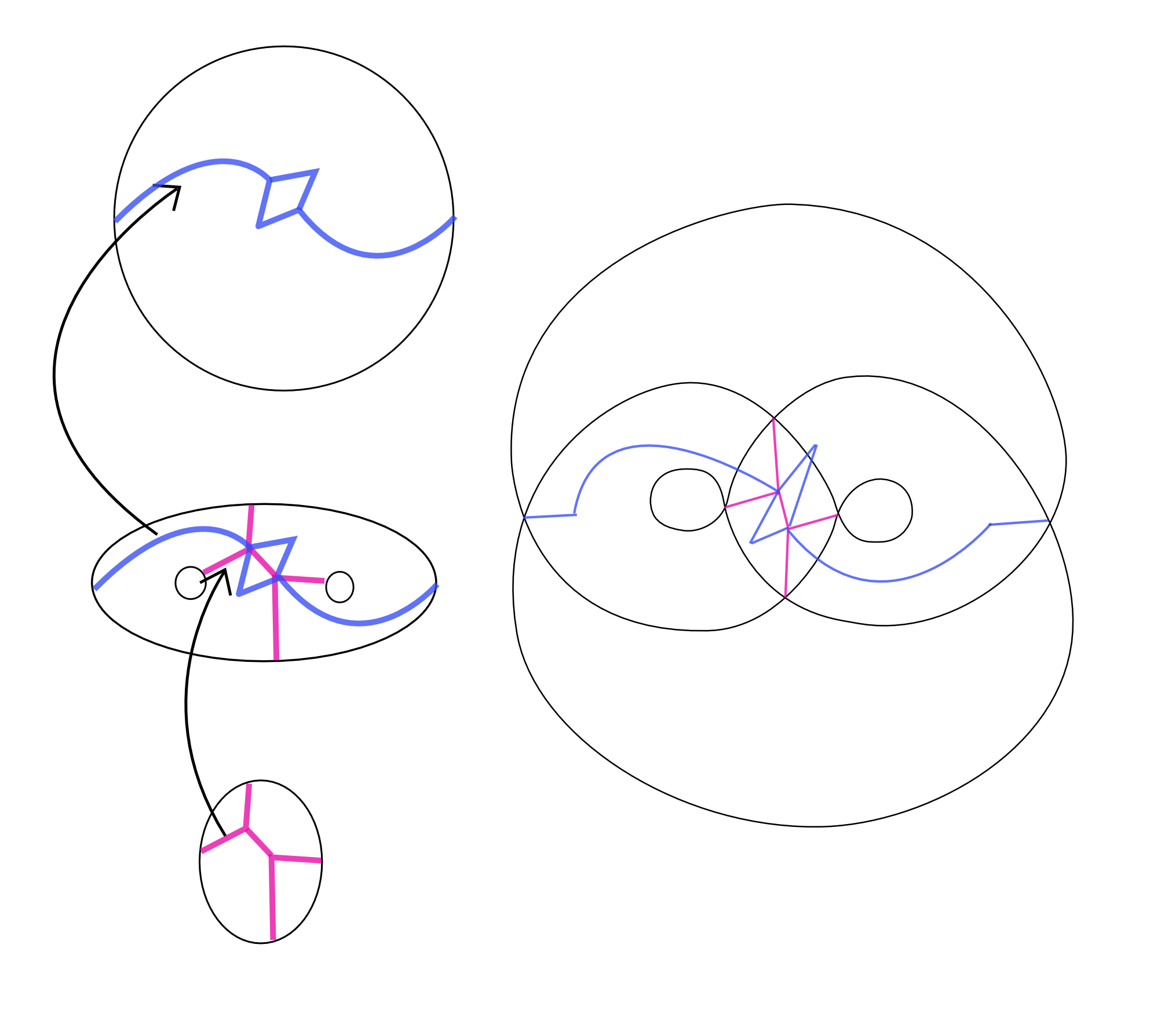}
	\includegraphics[width=0.45\textwidth]{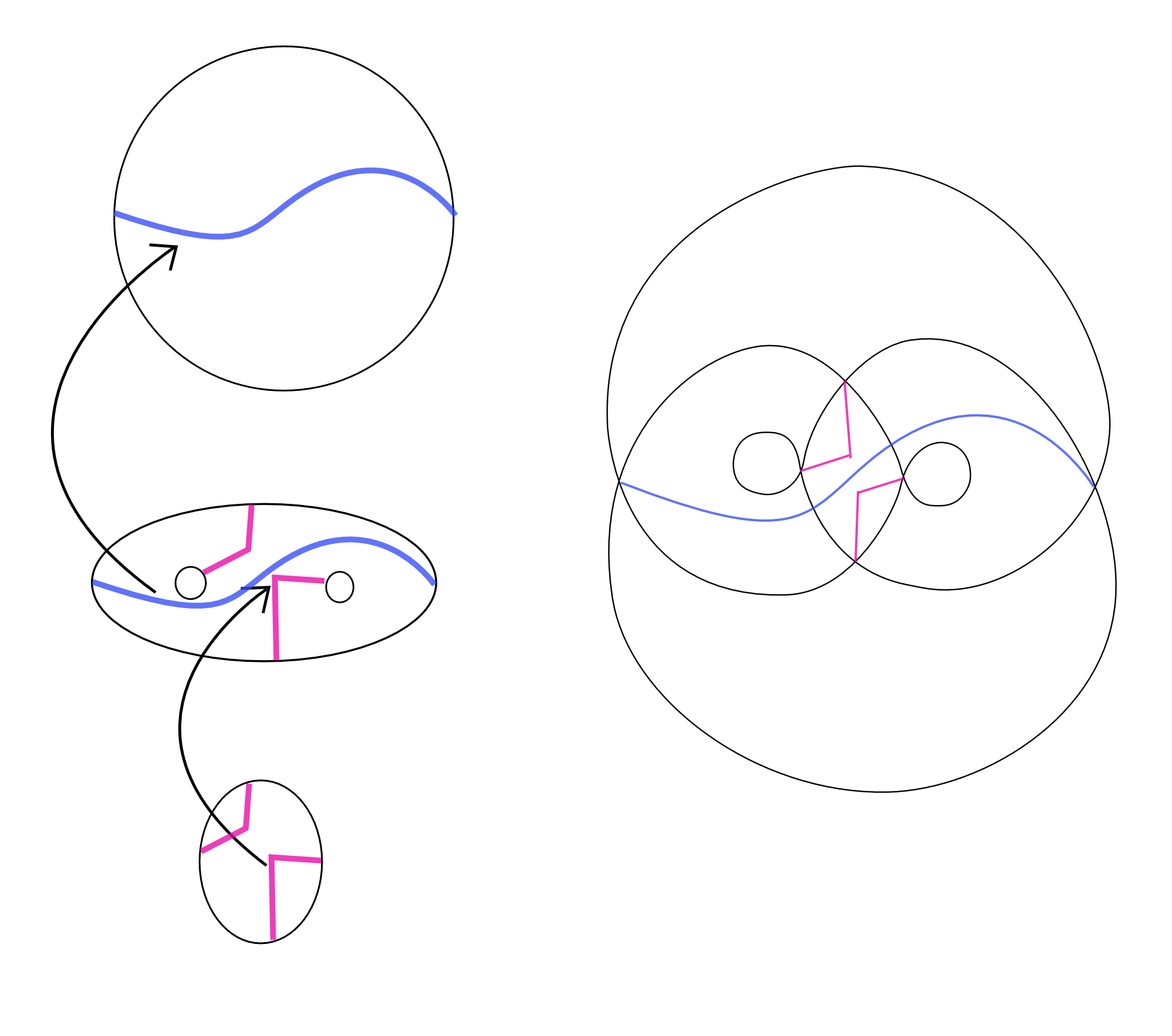}
	(c)
	\caption{Example $1$: (a) $\zeta_3$, (b) $\zeta_2$, (c) $I(\zeta)$, and the construction of inscribed maps.}
\end{figure}

\begin{figure}\label{surface graph}
	\centering 
	\includegraphics[width=0.8\textwidth]{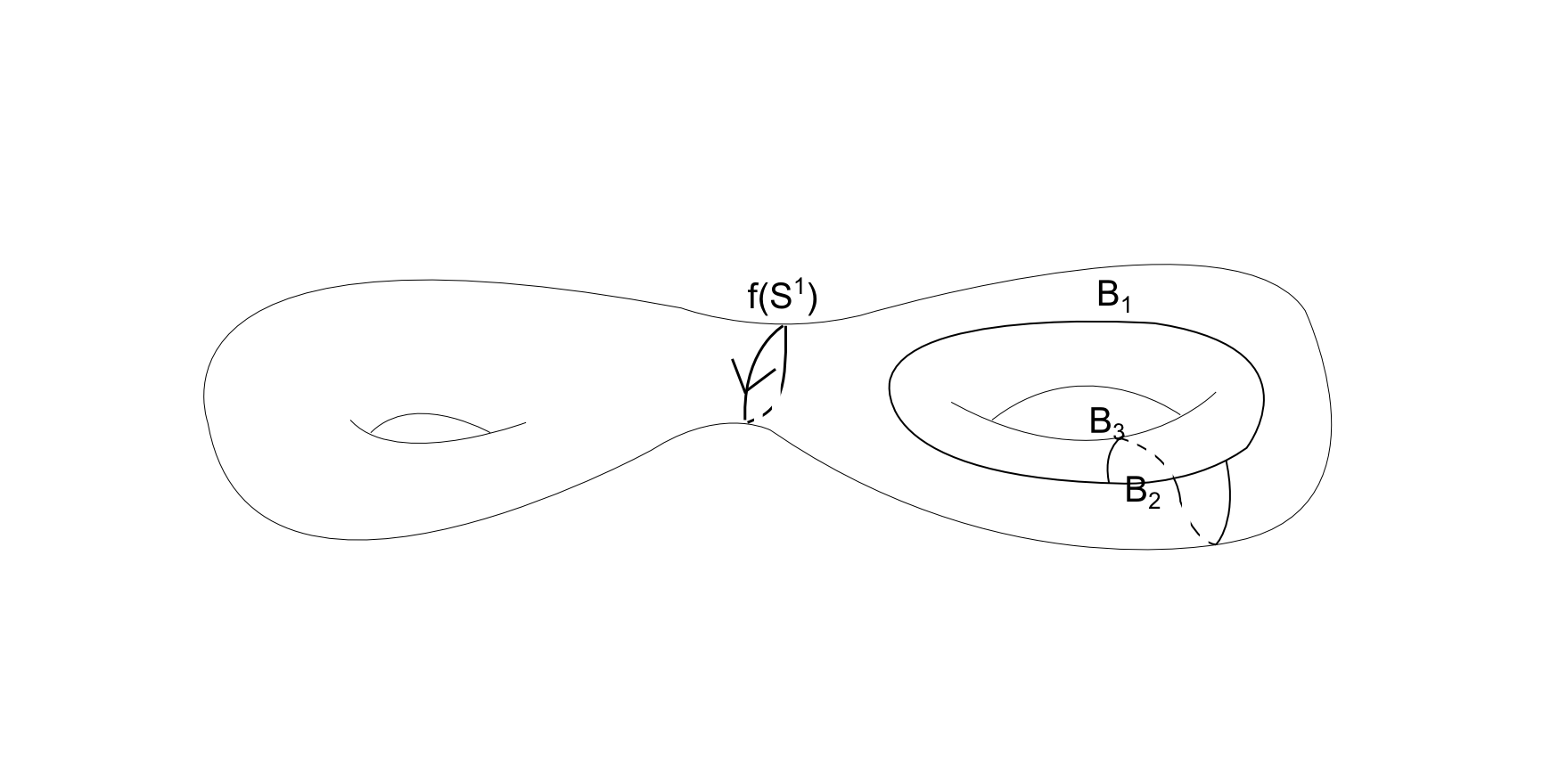}
	\caption{Example $2$.}
\end{figure}

We summarize the result without fixing a normal numbering as follows.

Fix a closed oriented surface $\Sigma$ and a homologically trivial immersion $f: S^{1} \to \Sigma$.
Let $\psi_i$ denote the normal numbering of $f$ such that $\min \psi_i = i$,
$i \in \mathbb{Z}_{\geqslant 0}$.
Let $\zeta$ be an inscribed set of $(f,\psi_0)$.
We construct $\zeta_0, \zeta_{-1}, \zeta_{-2},\ldots$ as follows:
we similarly obtain $\zeta_{-i}$ from $\zeta_{-i+1}$ as obtaining $\zeta_k$ from $\zeta_{k+1}$ ($k \geqslant 1$) in Definition \ref{inscribed set}, $i \geqslant 0$.
Let $I(\zeta_{-i})$ ($i \geqslant 0$) be the set composed of each $\{(\tilde{H}_{-i},H_{-i}),\ldots,(\tilde{H}_0,H_0),(\tilde{H}_1,H_1),\ldots,(\tilde{H}_n,H_n)\} \in \zeta_{-i}$
such that $H_{-i} = \emptyset,$ $H_{k} \ne \emptyset$, for each $-i+1 \leqslant k \leqslant n$.
Now,
each element of
$I(\zeta_{-i})$ $(i \geqslant 0)$ is in correspondence with an inscribed map of $(f,\psi_{i+1})$,
which is an immersion of a connected oriented surface into $\Sigma$.
Recall that the extensions related to different normal numberings are inequivalent (or, see \hyperref[Frisch]{[8}, Lemma 5.3.2]).
Applying Theorem \ref{main} for all normal numberings to
get a bijection between the set of all equivalence classes of extensions of $f$ and $I(\zeta) \cup (\bigcup_{i \geqslant 0}I(\zeta_{-i}))$.

Our definitions and results can be extended to the cases of $\Sigma = \mathbb{R}^{2}$ or $\Sigma$ is a compact orientable surface with nonempty boundary.
In these cases,
since arbitrary extensions of $f$ (if exist) are related to a unique normal numbering,
we only need to consider one normal numbering (see Subsection \ref{subsection 2.1}).

\section{Examples}\ \label{section 8}

\begin{ep}\rm
Let $f: S^{1} \to \mathbb{R}^{2}$ be an immersion such that $f(S^{1})$ is a Milnor curve (Figure \ref{Milnor curve}).
Let $\psi$ be the normal numbering whose images are winding numbers.
Fix $\zeta$ an inscribed set of $(f,\psi)$ shown by Figure \ref{example 1} (a), (b) (assume $X \in \zeta_i$, $X = \{(\tilde{H}_i(X),H_i(X)),\ldots,(\tilde{H}_3(X),H_3(X))\}$. We only paint $H_i(X),\ldots,H_n(X)$ in the figure of it).
Then $I(\zeta)$ is shown by Figure \ref{example 1} (c).
The inscribed maps related to the elements of $I(\zeta)$ are also shown by it.
\end{ep}

\begin{ep}\rm
Let $\Sigma$ be a surface of genus $2$
and $f: S^{1} \to \Sigma$ the immersion shown by Figure \ref{surface graph}.
Assume $\psi$ is a normal numbering such that the left and right side of $f(S^{1})$ has image $2$ and $1$.
We set $\tilde{H}_2$ (see Figure \ref{surface graph}) that has $3$ edges $B_1,B_2,B_3$.
$H_2$ can be $B_1 \cup B_2$ or $B_1 \cup B_3$ or $B_2 \cup B_3$.
Hence there are $3$ inequivalent extensions exactly.
\end{ep}

\section{Acknowledgments}\ \label{section 9}

I thank Professor Shicheng Wang, since he introduced me the book \emph{Problems in Low-dimensional Topology} and intrigued my thinking through the discussion, and encouraging me insisting in my initial idea. I thank Professor Yi Liu for his much advice, especially in suggesting me to try applying my planar technique to the case of surface. I thank Professor Jiajun Wang, in pointing out the things I should do, and set the standard for me. I am grateful to them.


\begin{thebibliography}{10}
 	
 	\bibitem{Bai67}
 	K. D. Bailey, 
 	\newblock{Extending closed plane curves to immersions of the disk with n handles}, 
 	Trans. Amer. Math. Soc. 206 (1975), 1-24
 	\label{Bailey}
 	
    \bibitem{Bla67}
    S. J. Blank, 
    \newblock{Extending Immersions and regular Homotopies in Codimension 1}, PhD thesis, Brandeis University, 1967 \label{Blank}
    
    \bibitem{Cal09}
    D. Calegari, 
    \newblock{Faces of the scl-norm ball}, Geom. Topol. 13 (2009), 1313-1336 \label{Calegari}
    
    \bibitem{CW86}
    C. Curley, D. Wolitzer, 
    \newblock{Branched immersions of surfaces}, Mich. Math. J. 33 (1986), 134-144
    \label{Curley, Wolitzer}
    
    \bibitem{EM80}
    C. L. Ezell and M. L. Marx, 
    \newblock{Branched extensions of curves in orientable surfaces}, Trans. Amer. Math. Soc. 259 (1980), 515-532 \label{Ezell, Marx}
    
    \bibitem{Fra70}
    G. K. Francis, 
    \newblock{Extensions to the disks of properly nested plane immersions of the circle}, Mich. Math. J. 17 (1970), 377-383 \label{Francis}
    
    \bibitem{Fra73}
    G. K. Francis, 
    \newblock{Spherical curves that bound immersed discs}, Proc. Amer. Math. Soc. 41 (1973), 87-93
    \label{Francis2}
    
    \bibitem{Fri10}
    D. Frisch, 
    \newblock{Classifications of Immersions which are Bounded by Curves in Surfaces}, PHD thesis, TU Darmstadt, 2010
    \label{Frisch}
    
    \bibitem{Kir95}
    R. Kirby,
    \newblock{Problems in Low-Dimensional Topoplogy}, (In R.Kirby, editor) Geometric topology (Athens, GA, 1993), AMS/IP Stud. Adv. Math. 2, Amer. Math. Soc., Providence, RI (1997) 35-473
    \label{Kirby}
    
    \bibitem{MC93}
    M. McIntyre, G. Cairns,
    \newblock{A new formula for winding number},
    Geom. Dedicata 46 (1993), 149-160 
    \label {McIntyre, Cairns}
    
    \bibitem{McI97}
    M. McIntyre, 
    \newblock{Bounding immersed curves}, Topology Appl. 78(1997) 251-267
    \label{McIntyre}
    
    \bibitem{Poe69}
    V. Poenaru, 
    \newblock{Extensions des immersions en codimension 1}, Bourbaki seminar, 1967/1968 (French), Exp. No. 342, Benjamin, New York, 1969
    \label{Poenaru}
    
    \bibitem{Tit61}
    C. J. Titus,
    \newblock{The conbinatorial topology of analytic functions on the boundary of a disk}, Acta Math. 106 (1961), 45-64
    \label{Titus}

\bibitem{Spr01}
D. Spring, 
\newblock{The Golden Age of Immersion Theory in Topology: 1959-1973. A mathematical
survey from a historical perspective}, Bull. Amer. Math. Soc. 42 (2005), 163-180
\label{Spring}

 \end{thebibliography}
\end{document}